\documentclass[11pt]{article}
\usepackage[utf8]{inputenc}
\usepackage[english]{babel}
\usepackage{geometry}
\usepackage{svg}

\usepackage{enumitem}

\usepackage{amssymb}   
\usepackage{stmaryrd}
\usepackage{amsmath}
\usepackage{lipsum}
\usepackage{amsmath,amscd}
\usepackage{dsfont}

\usepackage{pgf}  		
\usepackage{amsthm} 	
\usepackage{thmtools} 	

\usepackage[pagewise]{lineno}

\usepackage{wrapfig}			
\usepackage{standalone}
\usepackage{tikz,pgfplots}
\usepackage{tikz-cd}
\pgfplotsset{compat=1.15}
\usepackage{mathrsfs}
\usetikzlibrary{arrows}
\usetikzlibrary{decorations.markings}
\usepackage[framemethod=tikz]{mdframed}
\usepackage{adjustbox}
 
\usepackage{caption}

\usepackage{hyperref}
\usepackage{xcolor}

\usepackage{enumitem}
\usepackage{blindtext}
\usepackage{comment}
\usepackage{listings}

\usepackage{mathtools}
\newcommand{\defeq}{\vcentcolon=}

\definecolor{shadegray}{gray}{0.96} 	

 \newtheorem {definition} {Definition} [subsection] 
  \newtheorem {theorem}[definition] {Theorem} 
  \newtheorem {corollary}[definition]  {Corollary}
\newtheorem {prop}[definition] {Proposition}
\newtheorem {lemma}[definition] {Lemma} 
\newtheorem {remark}[definition] {Remark} 
\newtheorem {example}[definition] {Example} 
\numberwithin{equation}{section}
\newcommand{\theoremcolorbox}[2]{\def\theoremframecommand{\fcolorbox{#1}{#2}}}
\newtheorem{hyp}{Hypothesis}
\theoremcolorbox{black}{cyan}

\newcommand{\encad}[1]{%
\fbox{\begin{minipage}[t]{0.92\linewidth}%
#1\end{minipage}}}

\newcommand*{\EnsembleQuotient}[2]%
{\ensuremath{%
    #1/\!\raisebox{-.65ex}{\ensuremath{#2}}}}
    
\newcommand*{\RZ}{\EnsembleQuotient{\mathbb{R}}{\mathbb{Z}}}
\newcommand*{\rz}{\mathbb{R}/\mathbb{Z}}

\newcommand*{\fonction}[5]{
\begin{align*}
#1 : \left\{\begin{array}{lcl}  #2 &\to      &#3\\
                                #4    &\mapsto  &#5
\end{array} \right. .      
\end{align*} 
}

\newcommand*{\fonctionv}[5]{
\begin{align*}
#1 : \left\{\begin{array}{lcl}  #2 &\to      &#3\\
                                #4    &\mapsto  &#5
\end{array} \right. ,      
\end{align*} 
}
\newcommand{\uo}[3]{\underset{#1}{\overset{#2}{#3}}}

\title{Galton-Watson processes in dynamical environments}
\author{Thomas Morand\thanks{Université Paris-Saclay, CNRS, Laboratoire de mathématiques d’Orsay, 91405, Orsay, France\\ Email: thomas.morand@universite-paris-saclay.fr}\\\textit{Paris Saclay University}}
\date{2024}

\begin{document}  

\renewcommand{\proofname}{Proof}
\renewcommand\refname{References}
\renewcommand*\contentsname{Table of contents}
\newcommand{\argmin}{\text{argmin}}
\maketitle

\begin{abstract}

We define a model of Galton Watson processes in dynamical environments where the environment evolves according to a dynamical system $(\mathbb{X},T)$. Three behaviours are possible: uniformly subcritical, critical, and uniformly supercritical. We study the extinction probability $q$ in the uniformly supercritical case. In particular, we investigate the regularity of this application as a function of the environment $x$. In the critical case, we study the set of bad environments $N$ (where the probability of extinction is one), which is $T$-invariant. We give its Hausdorff dimension in some cases.
\end{abstract}

\section*{Introduction}
We define a model of Galton-Watson processes in dynamical environments. These are Galton-Watson processes where the reproduction law changes between generations. We give a discrete-time dynamical system defined on a compact space $(\mathbb{X},T)$ and a continuous law of reproduction $\mu$. Defined on $\mathbb{X}$, the reproduction law at generation $n$ is $\mu_{T^nx}$. 
These processes are special cases of Galton-Watson processes in varying environments studied in particular in \cite{MR0368197,MR0365733,MR2384553,MR4094390}.

To study this model, we use the probability generating function of the law of reproduction $\mu$, defined on $\mathbb{X}\times [0,1]$ as:
\begin{align*}
    \varphi(x,s) \defeq \sum_{k=0}^{+\infty} \mu_x(k)s^k.
\end{align*}
The extinction probability $q$ is a function defined on $\mathbb{X}$. When $q$ is equal to one, there is almost certain extinction of the process, and we say that $x\in\mathbb{X}$ is a bad environment. We will then denote $N$ the set of bad environments. However, since $N$ is a $T$-invariant set, using the point of view of ergodic theory, we can ask under what condition the measure of $N$ by an ergodic probability $T$ on $\mathbb{X}$ is zero or one. The work of Athreya and Karlin \cite{MR0298780}, under some integrability assumptions, answers this question. When $\nu$ be an $T$-ergodic probability on $\mathbb{X}$,\begin{align*}
    \nu(N)=0 \text{ if and only if }\mathbb{E}_\nu[\log m(x)]>0,
\end{align*}
where $m(x)$ is the expectation of the law $\mu_x$ for $x\in\mathbb{X}$.

Three behaviours are possible. The process is: \begin{itemize}
 \item \textbf{uniformly subcritical} if the set $N$ is of full measure according to all ergodic probability measures,
 \item \textbf{critical} if the set $N$ is of measure zero for some ergodic probability measure and of full measure for another,
 \item \textbf{uniformly supercritical} if the set $N$ is of measure zero according to all ergodic probability measures.
\end{itemize}

In the critical case, when $(\mathbb{X},T)$ is an expanding map of the circle or a two-dimensional Anosov, by adapting the work of Keller and Otani \cite{MR1418993,MR3170606}, we can look at the Haussdorf dimension of the set $N$ (Theorem~\ref{thm4}). This theorem allows us to affirm that the probability of survival is positive for $\nu$-almost all $x\in\mathbb{X}$ if $\nu$ has a Hausdorff dimension larger than that of $N$.

The probability of extinction $q$ is a solution of the functional equation: \begin{align*}
     q(x)=\varphi(x,q(Tx)).
\end{align*}
We then say that $q$ is an invariant graph. The regularity of invariant graphs of skew product systems has been studied in \cite{MR1605989,MR1677161,MR1898800,MR3816739} and the regularity of the probability of extinction of Galton-Watson processes in varying environments has been studied in \cite{MR2384553}. In the supercritical case and assuming some assumptions of integrability, one may wonder whether $x\in\mathbb{X}\mapsto q(x)$ inherits the regularity of $x\in\mathbb{X}\mapsto\mu_x$. Continuity will be well preserved by $q$ (Theorem~\ref{thm3}). An essential tool in the proof is the semi-uniform ergodic theorem \cite[Theorem 1.9]{MR1734626}. We also get bounds on the Hölder regularity of $q$ (Theorem~\ref{thm2.6}).
\newpage
\tableofcontents
\newpage
 \section{Model and results}

In this section, we will present the different Galton-Watson models found in the literature, introduce the Galton-Watson model in dynamical environments, present the results of this article, and observe the link between Galton-Watson processes in dynamical environments and skew products.

\subsection{Galton-Watson processes: classical case and random environments}\label{section1.1}

We will begin by defining Galton-Watson processes in three well-known cases: the classical case, in variable environments, and in random environments. These models will enable us to define the Galton-Watson processes in dynamical environments.

\subsubsection{Galton-Watson processes}

A Galton-Watson process is a stochastic process used to describe the evolution of a population. It is a discrete-time process where each individual born in generation $n$
dies at time $n+1$ and produces a random number of offspring at time $n+1$ who live, die, and reproduce in the same way independently. In the classical case, the reproduction law governing the number of descendants of an individual (a random variable with support in $\mathbb{N}$) is the same for all individuals. We will assume that the population at generation zero always equals one. Bienaymé \cite{bienayme1845loi} in 1845, then Galton and Watson \cite{10.2307/2841222} in 1875 introduced this model to describe the evolution of surnames in a population.

Let $\mu \in \mathcal{P}(\mathbb{N})$ be a probability measure on $\mathbb{N}$. The Galton-Watson process associated with the law of reproduction $\mu$ is the sequence of random variables $(Z_n)_{ n \in \mathbb{N}}$ defined recursively by:
\begin{align}
\left\{\begin{array}{ll}
   Z_0 &= 1 ,\\
    Z_{n+1} &= \underset{k=1}{\overset{Z_n}{\sum}}Y_{n,k}\text{ for all }n\in\mathbb{N},\label{eq3}
    \end{array}\right. 
\end{align}
where $(Y_{n,k})_{ (n,k) \in \mathbb{N}^2}$ is a family of independent random variables such that $Y_{n,k}$ is distributed according to $\mu$ for all $n,k \in \mathbb{N}$. For all $n \in \mathbb{N}$, $Z_n$ is a random variable representing the size of the population at the $n$th generation. We define the extinction set as:
\begin{align*}
    \text{Ext} \defeq \bigcup_{n\geq 0}\{Z_n=0\}.
\end{align*} 
Let $q\defeq\mathbb{P}(\text{Ext})$ the probability of extinction. To calculate this, we introduce the generating function of a probability measure $\mu\in\mathcal{P}(\mathbb{N})$, defined by:
\begin{align*} 
    \varphi_\mu(s) = \sum_{k=0}^{+\infty} s^k\mu(k) \text{ for all }s\in[0,1].
\end{align*}

$\varphi_\mu$ is analytic on [0,1), convex (strictly convex if $\mu(\{0,1\})<1$), $\varphi_\mu(0)=\mu(0)$, and $\varphi_\mu(1)=1$. Additionally, $\mu$ has a moment of order $k\in\mathbb{N}^*$ if and only if the $k$-th derivative of $\varphi_\mu$ has a finite limit at 1. In particular, $m\defeq\varphi_\mu'(1)$ is the expectation of $\mu$.

Then, the probability of extinction $q$ is the smallest $s\in[0,1]$ such that $\varphi_{\mu}(s)=s$. In particular, if $m\in[0,1]$ and $\mu\neq\delta_1$, then $q=1$, and if $m \in(1,+\infty]$, then $q<1$.

\begin{figure}[!ht]
    \center
\includegraphics[scale=0.64]{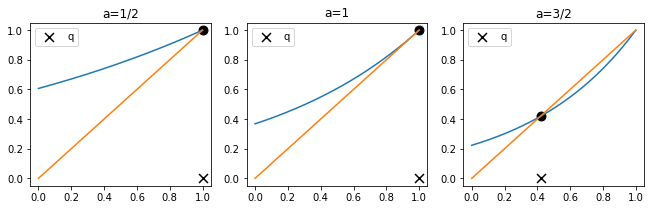}
\captionof{figure}{Probability-generating function for the Poisson distribution with parameter $a$, and the probability of extinction $q$ of the associated Galton-Watson process.} 
\end{figure}

\subsubsection{Galton-Watson processes in varying environments}

We now assume that the law of reproduction remains the same within a generation but evolves between generations. Let $(\mu_n)_{n\in\mathbb{N}}$ be a sequence of probability measures on $\mathbb{N}$. The Galton-Watson process in varying environments $(Z_n)_{n\in\mathbb{N}}$ associated with $(\mu_n)_{n\in\mathbb{N}}$ is defined recursively by Equation (\ref{eq3})
where $(Y_{n,k})_{ (n,k) \in \mathbb{N}^2}$ is a family of independent random variables such that $Y_{n,k}$ is distributed according to $\mu_n $ for all $n,k \in \mathbb{N}$. Jagers \cite{MR0368197} proposed a classification into supercritical, critical and subcritical regimes as in the classical case. Kersting \cite{MR4094390} gives conditions for almost certain extinction or strictly positive survival probability as a function of the first and second moments of the laws of reproduction. These processes have recently been the subject of renewed interest, see for example \cite{MR3371434,MR4019875,MR3969512,MR4388943}.

\subsubsection{Galton-Watson processes in random environments}

We assume that the sequence of environments is random and independent of reproduction. The most studied case is when $(\mu_n)_{n\in\mathbb{N}}$ is i.i.d.\ according to a probability law $\nu$ on $\mathcal{P}(\mathbb{N})$. 
We then have two levels of independent randomness: environments and offspring.

This model was introduced by Smith and Wilkinson \cite{MR0246380,MR0237006}. They discussed the probability of extinction. Athreya and Karlin \cite{MR0298780,MR0298781} reformulated and generalized the model, studied the probability of extinction, and gave limit theorems for the Galton-Watson processes. Geiger, Kersting, and Vatutin \cite{MR2123206} have studied the extinction time when extinction is almost certain. These processes are still being studied, see for example \cite{MR4259452,MR4721632,MR4779872,MR4740913}. More generally, Galton-Watson processes with $(\mu_n)_{n\in\mathbb{N}}$ a Markov process have also been studied \cite{MR0288885,MR2600128,MR3958440}.

\subsection{Galton-Watson processes: dynamical environments}
We define the model of Galton-Watson in dynamical environments from Galton-Watson processes in varying environments. We will then define the probability generating function $\varphi$, the probability of extinction $q$, and the set of bad environments $N$.
\subsubsection{Presentation}

In this article, we consider $(\mathbb{X},\mathcal{B}(\mathbb{X}),T)$ a topological discrete-time dynamical system, where:
\begin{itemize}
    \item $(\mathbb{X},d)$ is a compact metric space equipped with its Borel algebra $\mathcal{B}(\mathbb{X})$,
    \item $T:\mathbb{X}\to\mathbb{X}$ is a continuous map.
\end{itemize}
We denote by $\mathcal{P}_T(\mathbb{X})$ the set of $T$-invariant probabilities on $\mathbb{X}$ endowed with the weak*-topology, and by $\mathcal{E}_T(\mathbb{X})$ the set of $T$-ergodic probabilities on $\mathbb{X}$.

We equip $\mathcal{P}(\mathbb{N})$ with the topology generated by $\ell^1$ and its associated Borel algebra. Additionally, let:
\begin{align*}
    \mu:\left\{
    \begin{array}{lcl}
    \mathbb{X}&\to&\mathcal{P}(\mathbb{N})\\
    x&\mapsto&\mu_x
    \end{array}
    \right. 
\end{align*}
and assume (H\ref{hyp1}).

\encad{
\begin{hyp}[H\ref{hyp1}]\label{hyp1}~
\begin{enumerate}[label=\alph*)]
    \item $x\in\mathbb{X}\mapsto \mu_x$ is continuous.
    \item $\mu_x\neq\delta_0$ for all $x\in\mathbb{X}$.
\end{enumerate}
\end{hyp}}\medskip

For each $x\in\mathbb{X}$, let $(Z_n(x))_{n\in\mathbb{N}}$ be the Galton-Watson process in varying environments associated with $(\mu_{T^nx})_{n\in\mathbb{N}}$. In other words, $(Z_n(x))_{n\in\mathbb{N}}$ is defined recursively by Equation (\ref{eq3})
where $(Y_{n,k})_{ (n,k) \in \mathbb{N}^2}$ is a family of independent random variables such that for all $(n,k) \in \mathbb{N}^2$, $Y_{n,k}$ distributed according to $\mu_{T^n x} $.
The family of random variables $(Z_n(x))_{n\in\mathbb{N}}$ is the Galton-Watson process associated with the discrete-time dynamical system $(\mathbb{X},\mathcal{B}(\mathbb{X}),T)$, the reproduction law $\mu$, and the point $x\in\mathbb{X}$.

\begin{example}This example is related to the other Galton-Watson models presented in Subsection~\ref{section1.1}.
    \begin{itemize}
        \item If $x\in\mathbb{X}$ is a fixed point of the transformation $T$, then the process $(Z_n(x))_{n\in\mathbb{N}}$ is a classical Galton-Watson process with a reproduction law $\mu_x$.
        \item If $(\mathbb{X},T)$ is a Bernoulli shift, $\mu$ depends only on the first variable, and $x\in\mathbb{X}$ is chosen according to an invariant distribution, then the process $(Z_n(x))_{n\in\mathbb{N}}$ is a Galton-Watson process in random i.i.d.\ environments.
    \end{itemize}
\end{example}

Example~\ref{ex1} will be used several times in the rest of this article to illustrate our results.

\begin{example}\label{ex1} Let:\begin{itemize}
    \item $\mathbb{X}\defeq\RZ$ with the Borel algebra $\mathcal{B}(\RZ)$ and the usual distance on the circle,
    \item $ T\defeq \left\{
    \begin{array}{lcl}\RZ&\to&\RZ\\
    x &\mapsto &2x \text{ modulo } 1\end{array}\right.,$
    \item For all $\lambda\in\mathbb{R}$, $\mu_{\lambda,.}\defeq\left\{
    \begin{array}{lcl}\RZ&\to&\mathcal{P}(\mathbb{N})\\
    x &\mapsto &\text{Pois}\left(e^{\lambda-\cos{(2\pi x)}}\right)\end{array}
\right. $.
\end{itemize}
For each $\lambda\in\mathbb{R}$, this defines a Galton-Watson process in dynamical environments which satisfies (H\ref{hyp1}). 
In this example, we will add an index $\lambda$ to all the objects defined for the study of this model to show the dependency on the parameter $\lambda$.
\end{example}

\subsubsection{Extinction and generating function}\label{section3}

In dynamical environments, the law of reproduction depends on the environment, so the generating functions of the law of reproduction depend on it too.
Propositions of this sub-subsection are proven in Section~\ref{section2}.

\begin{definition}~
The probability generating function of the law of reproduction $\mu$ is defined on $\mathbb{X}\times [0,1]$ as:
\begin{align*}
    \varphi(x,s) \defeq \sum_{k=0}^{+\infty} \mu_x(k)s^k = \mathbb{E}[s^{Y(x)}],
\end{align*}
where $Y(x)$ is distributed according to $\mu_{x} $.

For any $n \in \mathbb{N}$, the probability generating function of the distribution of $Z_n(x)$ is:
\begin{align*}
    \varphi^{(n)}(x,s) \defeq \sum_{k=0}^{+\infty} \mu_x^{(n)}(k)s^k= \mathbb{E}[s^{Z_n(x)}],
\end{align*}
where $x\in\mathbb{X}$, $s\in[0,1]$, and $\mu_x^{(n)}$ is the distribution of $Z_n(x)$.
\end{definition}

For any $x \in \mathbb{X}$ and $s \in [0,1]$,
\begin{align*}
\varphi^{(0)}(x,s)&=s,\\
    \varphi^{(1)}(x,s) &=\varphi(x,s).
\end{align*}

More generally, we relate the probability generating function of the population at successive times:

\begin{prop}\label{prop5}
For any $x \in \mathbb{X}$, $n,k\in \mathbb{N}$, and $s \in [0,1]$:
    \begin{align}\label{eq5}
 \varphi^{(n+k)}(x,s) = \varphi^{(k)}(x,\varphi^{(n)}(T^kx,s)).
    \end{align}
\end{prop}

The probability generating function allows us to define the probability of extinction.

\begin{definition}
We define the probability of extinction $q$ for all $x\in\mathbb{X}$, by:\begin{align*}
    q(x)\defeq\underset{n\to\infty}{\lim}\nearrow \varphi^{(n)}(x,0)=\underset{n\to\infty}{\lim}\nearrow\mathbb{P}(Z_n(x)=0).
\end{align*}
\end{definition}

A function that verifies the functional equation of Proposition~\ref{prop1} is called an invariant graph of our system. The constant function equal to 1 is a solution to this functional equation.

\begin{prop}\label{prop1}
Assume (H\ref{hyp1}), for each $x\in\mathbb{X}$:
\begin{align}
    q(x)=\varphi(x,q(Tx)).
\end{align}
\end{prop}

As in the classical case, the expectation of the laws of reproduction allows us to express the probability of extinction of the process.

\begin{definition}
  For all $x\in\mathbb{X}$, let:
\begin{align*}
    m(x)\defeq \partial_s\varphi(x,1)\in (0,+\infty].
\end{align*}
\end{definition} 
For all $x\in\mathbb{X}$, $m(x)$ is the expectation of $\mu_x$.

\begin{definition}
The set of bad environments is $N\defeq\{x\in\mathbb{X}:q(x)=1\}$.
\end{definition}

The set $N$ is the set of environments where the process $\{Z_n(x),n\in\mathbb{N}\}$ almost surely goes extinct, i.e., where $\mathbb{P}\left(\underset{n\geq 0}{\bigcup}\{Z_n(x)=0\}\right)=1$.

\subsection{Main results}

We present the main results of this article on the regularity of the invariant graph $q$ and on the measure and the topology of the set of bad environments $N$.

\subsubsection{The set of bad environments}

\paragraph{Measure of the set of bad environments}~\\
The set $N$ is difficult to describe precisely (see Lemma~\ref{lem6}). We will begin our study by using the point of view of ergodic theory and, therefore, by calculating the measure of this set by an invariant measure (and even ergodic without losing generality through the ergodic decomposition). The set $N$ is $T$-invariant (see Proposition~\ref{cor2}), so its measure according to an ergodic measure equals $0$ or $1$.
The results of Athreya and Karlin ~\cite[Corollary~1]{MR0298780} and ~\cite[Theorem~3]{MR0298780} apply if the environment is a stationary and ergodic process, which is the case if we fix an ergodic measure. Corollary~\ref{cor4} and Theorem~\ref{thm2} (which are results of Athreya and Karlin adapted in the context of our model) provide criteria to determine whether or not there is almost sure extinction of the process in the case where $\nu$ is an ergodic measure. Theorem~\ref{thm1} is a direct corollary of these results in the case of our model. The proof is in Subsection~\ref{section3.1}. 

\begin{theorem}\label{thm1}
    Assume (H\ref{hyp1}). Let $\nu\in\mathcal{E}_T(\mathbb{X})$. Then $\nu(N)=0$ if and only if \newline$\mathbb{E}_\nu[\log m(x)]>0$.
\end{theorem}
We can define:\begin{align*}
    \lambda_{\min}\defeq& \inf_{\nu\in\mathcal{P}_T(\mathbb{X})} \mathbb{E}_\nu[\log m],\\
    \lambda_{\max}\defeq& \sup_{\nu\in\mathcal{P}_T(\mathbb{X})} \mathbb{E}_\nu[\log m].
\end{align*}

By Theorem~\ref{thm1}, the process is uniformly subcritical if $\lambda_{\max}\leq0$, critical if $\lambda_{\min}\leq0<\lambda_{\max}$, and uniformly supercritical if $\lambda_{\min}>0$. Determining the regime of a Galton-Watson process in dynamical environments is equivalent to solving an ergodic optimisation problem. These questions have been studied by \cite{MR4000508} for example.

\paragraph{Dimension of the set of bad environments}~\\
 Theorem~\ref{thm4} and Theorem~\ref{thm5} give the Hausdorff dimension of the set $N$ in two types of system. These theorems are adapted from \cite[Theorem~2]{MR3170606} in a slightly different model. The similarities and differences between the two models are developed in Subsection~\ref{section1}. Theorem~\ref{thm4} and Theorem~\ref{thm5} are proved in Subsection~\ref{section2.2}.

\encad{
\begin{hyp}[H\ref{hyp5}]\label{hyp5}~
    \begin{enumerate}
      \item There exists $\alpha\in (0,1]$ such that $x\in\mathbb{X}\mapsto \mu_x$ is $\alpha$-Hölder continuous for the $\ell^1$ norm on $\mathcal{P}(\mathbb{N})$.
    \item $\mu_x(\{0,1\})<1$ for all $x\in\mathbb{X}$.
    \item Critical case: $\lambda_{\min}<0<\lambda_{\max}$.
    \item The second moment of $(\mu_x)_{x\in\mathbb{X}}$ is uniformly bounded: $\underset{x\in\mathbb{X}}{\sup}~\underset{k=0}{\overset{+\infty}{\sum}}k^2\mu_x(k)<+\infty$.
\end{enumerate} 
\end{hyp}}

\begin{theorem}\label{thm4}
    Assume (H\ref{hyp5}), that $\mathbb{X}$ is a two-dimensional compact Riemannian manifold, and that $T$ a topologically mixing $\mathcal{C}^2$-Anosov diffeomorphism. Denote\\ $T_x\mathbb{X}=E^s(x) \bigoplus E^u(x)$ the splitting decomposition of the tangent fibre over $x\in\mathbb{X}$ into its stable and unstable subspaces.
    \begin{align}
        dim_H(N)=
    \left\{\begin{array}{ll}
    &2\text{ if } \mathbb{E}_{\nu_{SRB}}[\log m]\leq 0,\\
    &\underset{\nu\in\mathcal{P}_T(\mathbb{X})}{\max}\left\{\frac{h_T(\nu)}{\nu(\log \lVert dT|E^u\rVert)}:\int_{\mathbb{X}}\log m(x)\, \mathrm{d}\nu(x)=0\right\}+1 \text{ else.}
    \end{array}
\right.
    \end{align}
    $\nu_{SRB}$ is the unique Sinai-Ruelle-Bowen measure of $T$ characterized by the variational principle (see \cite[Section 4B]{MR0442989}): \begin{align*}
        &h_T(\nu_{SRB})-\nu_{SRB}(\log \lVert dT|E^u\rVert)=\sup_{\nu\in\mathcal{P}_T(\mathbb{X})}\big(h_T(\nu)-\nu(\log \lVert dT|E^u\rVert)\big)=0.
    \end{align*}
\end{theorem}

\begin{theorem}\label{thm5}
Assume (H\ref{hyp5}), that $\mathbb{X}=\RZ$, and that $T$ is a $\mathcal{C}^2$ uniformly expanding transformation (there exists $\kappa>1$ such that for all $x\in\RZ$, $|T'(x)|>\kappa$).
    \begin{align}
        dim_H(N)=
    \left\{\begin{array}{ll}
    &1\text{ if } \mathbb{E}_{\nu_{Leb}}[\log m]\leq 0,\\
    &\underset{\nu\in\mathcal{P}_T(\rz)}{\max}\left\{\frac{h_T(\nu)}{\nu(\log T')}:\int_{\mathbb{R}/\mathbb{Z}}\log m(x)\, \mathrm{d}\nu(x)=0\right\} \text{ else.}
    \end{array}
    \right.
    \end{align}
 $\nu_{Leb}$ is the unique invariant probability measure absolutely continuous with respect to the Lebesgue measure (see \cite{MR0245761}). Moreover (see \cite[Exercise 2.7]{MR1793194}), 
 \begin{align*}
        &h_T(\nu_{Leb})-\nu_{Leb}(\log T')=\sup_{\nu\in\mathcal{P}_T(\mathbb{R}/\mathbb{Z})}\big(h_T(\nu)-\nu(\log T')\big)=0.
    \end{align*} 
\end{theorem}

Let $\nu\in\mathcal{P}(\mathbb{X})$ such that $dim_H(\nu)>\dim_H(N)$. Then $\nu(N)=0$ (see for example \cite[Proposition~10.2]{MR1102677}). Thus, for $\nu$-almost all $x\in\mathbb{X}$, the probability of survival of the process $(Z_n(x))_{n\in\mathbb{N}}$ is positive.

\subsubsection{Regularity of the invariant graph}

 The question arises whether the regularity of the reproduction law $\mu$ leads to regularity in the invariant graph $q$ in the uniformly supercritical case.
The Hölder regularity of invariant graphs of skew product systems has been studied in \cite{MR1605989,MR1677161} but with an invertible transformation. In \cite{MR1898800}, the authors study the $\mathcal{C}^k$ regularity for a hyperbolic dynamical system. Here the main tool to obtain the continuity of the invariant graph are the semi-uniform ergodic theorem \cite[Theorem 1.9]{MR1734626} and the convergence of the probability generating function of the size of the population \cite[Theorem 5]{MR0298780}.

\begin{remark}
    In the case of Example~\ref{ex1}, the dependence on $\lambda$ of the probability of extinction inherits the regularity in $\lambda$ of the probability generating function \cite[Theorem 2.2]{MR2384553}.
\end{remark}

\paragraph{Continuity in the supercritical case}~\\
Theorem~\ref{thm3} allows us to see how continuity is preserved by $q$ in the supercritical case.

\encad{
\begin{hyp}[H\ref{hyp2}]\label{hyp2}~
\begin{enumerate}[label=\alph*)]
    \item $x\in\mathbb{X}\mapsto \mu_x$ is continuous.
    \item $\mu_x\notin\{\delta_0,\delta_1\}$ for all $x\in\mathbb{X}$.
    \item For all $\nu\in\mathcal{E}_T(\mathbb{X})$, $\mathbb{E}_\nu[\log m(x)]>0$ (uniformly supercritical case).
    \item There exists $\Tilde{\mu}$ a positive measure on $\mathbb{N}$, which has a first moment, and which stochastically dominates all the $(\mu_x)_{x\in\mathbb{X}}$: for all $k\in\mathbb{N}$ and $x\in\mathbb{X}$, $
    \mu_x([k,+\infty))\leq \Tilde{\mu}([k,+\infty))$.\label{h1}
\end{enumerate}
\end{hyp}}\medskip

\begin{theorem}\label{thm3}
Under (H\ref{hyp2}), $q$ is continuous.
\end{theorem}
Theorem~\ref{thm3} is proved in Subsection~\ref{section4.1}.

\begin{remark}
  The condition \textquotedblleft There exists $\varepsilon>0$ such that the $(1+\varepsilon)$ moment of $(\mu_x)_{x\in\mathbb{X}}$ is uniformly bounded: $\underset{x\in\mathbb{X}}{\sup}~\underset{k=0}{\overset{+\infty}{\sum}}k^{1+\varepsilon}\mu_x(k)<+\infty$\textquotedblright~implies the condition (H\ref{hyp2})\ref{h1}. 

    Indeed, assume this condition.
    Let $\Tilde{\mu}$ the measure on $\mathbb{N}$ defined by:
    \begin{align*}
        \Tilde{\mu}(k)=\sup_{x\in\mathbb{X}}\mu_x([k,+\infty))-\sup_{x\in\mathbb{X}}\mu_x([k+1,+\infty))\geq 0 \quad \forall k \in\mathbb{N}.
    \end{align*}
    Let $k\in\mathbb{N}$. By Markov's inequality and because $\underset{x\in\mathbb{X}}{\sup}~\underset{k=0}{\overset{+\infty}{\sum}}k^{1+\varepsilon}\mu_x(k)<+\infty$,
    \begin{align*}
        \Tilde{\mu}([k,+\infty))&=\sup_{x\in\mathbb{X}}\mu_x([k,+\infty))-\lim_{i\to +\infty}\sup_{x\in\mathbb{X}}\mu_x([i,+\infty))\\
        &=\sup_{x\in\mathbb{X}}\mu_x([k,+\infty)).
    \end{align*}
    Moreover, $\Tilde{\mu}$ has a first moment:
    \begin{align*}
        \sum_{k=1}^{+\infty}k\Tilde{\mu}(k)&=\sum_{k=1}^{+\infty}\sup_{x\in\mathbb{X}}\mu_x([k,+\infty))\\
        &\leq \underset{x\in\mathbb{X}}{\sup}\underset{i=1}{\overset{+\infty}{\sum}}i^{1+\varepsilon}\mu_x(i) \sum_{k=1}^{+\infty} \frac{1}{k ^{1+\varepsilon}} \text{ by Markov's inequality,}
        \\&<+\infty.
    \end{align*}   
\end{remark}

\paragraph{Hölder continuity in the supercritical case}~\\
Theorem~\ref{thm2.6} gives an estimate on the Hölder regularity of $q$.

\begin{definition} The Lyapunov exponent of the transformation $T$ is defined as:
   \begin{align*} 
    \lambda_u&\defeq\lim_{n\to\infty}\frac{1}{n}\log(\lVert T^n\rVert_{Lip})\in  \mathbb{R}\cup\{-\infty,+\infty\}.
  \end{align*} 
  The Lyapunov exponent of the invariant graph $q$ in the fibre is defined as:
  \begin{align}
      \lambda_F&\defeq\lim_{n\to\infty}\frac{1}{n}\sup_{x\in\mathbb{X}}\log(\partial_s\varphi^{(n)}(x,q(T^n x))).\label{eq7}
  \end{align}
\end{definition}

The Lyapunov exponent $\lambda_u$ is well defined (and different from $+\infty$ when $T$ is Lipschitz) by Fekete's subadditive lemma because $\lVert .\rVert_{Lip}$ is an algebra norm and $\lambda_F$ is well defined by Lemma~\ref{lem10}. $\lambda_u$ controls the speed at which two orbits move apart. In opposition, $\lambda_F$ allows us to control the speed at which the probability of extinction at the $n$th generation converges towards the invariant graph $q$. The Lyapunov exponent of the invariant graph $\mathds{1}$ in the fibre is $\lambda_{\max}$.\medskip

\encad{Let $\alpha\in(0,1]$.
\begin{hyp}[H\ref{hyp3}($\alpha$)]\label{hyp3}~
\begin{enumerate}[label=\alph*)]
    \item $\mu_x\notin\{\delta_0,\delta_1\}$ for all $x\in\mathbb{X}$.
    \item For all $\nu\in\mathcal{E}_T(\mathbb{X})$, $\mathbb{E}_\nu[\log m(x)]>0$ (uniformly supercritical case).
    \item There exists $\Tilde{\mu}$ a positive measure on $\mathbb{N}$ which has a first moment such that for all $k\in\mathbb{N}$ and $x\in\mathbb{X}$, $
    \mu_x([k,+\infty))\leq \Tilde{\mu}([k,+\infty))$.
    \item $x\in\mathbb{X}\mapsto \mu_x$ is in $\mathcal{C}^{0,\alpha}(\mathbb{X})$ for the $\ell^1$ norm ($\alpha$-Hölder continuous).
    \item $T$ is Lipschitz.\label{h2}
    \item $\lambda_u\leq 0$ or $ \left(\lambda_u>0 \text{ and }\alpha<-\frac{\lambda_F}{\lambda_u}\right)$.\label{h3}
    \item $q(x)>0$ for all $x\in\mathbb{X}$.\label{h5}
\end{enumerate} 
\end{hyp}}\medskip

Let $\alpha\in(0,1]$. Then (H\ref{hyp3}($\alpha$)) $\implies$ (H\ref{hyp2}) $\implies$ (H\ref{hyp1}).

Under (H\ref{hyp3}($\alpha$)), $\lambda_F$ is negative (see Lemma~\ref{lem12}). H\ref{hyp3}($\alpha$)\ref{h5} is equivalent to \textquotedblleft$\{x\in\mathbb{X}:\mu_x(0)=0\}$ does not contain any $T$-invariant subsets\textquotedblright.

The ratio $-\frac{\lambda_F}{\lambda_u}$ compares the convergence speed towards the invariant graph and the separation of two orbits by $T$. This ratio controls the Hölder seminorm of $\varphi^{(n)}$. In the supercritical case, under some integrability assumptions, $q$ is continuous, but we need (H\ref{hyp3}($\alpha$))\ref{h2} and \ref{h3} to have the Hölder continuity. In \cite{MR1605989,MR1677161}, the author also uses a ratio of Lyapunov exponents to control the Hölder regularity, but this depends on the inverse of the transformation. In Lemma~\ref{lem12}, we prove that $\lambda_F$ is negative, the case where it is equal to zero was studied in \cite{MR3816739}.

\begin{theorem}\label{thm2.6}
   Let $\alpha\in(0,1]$. Under (H\ref{hyp3}($\alpha$)), $q$ is $\alpha$-Hölder continuous.
\end{theorem}

Corollary~\ref{thmCN} asserts that the convergence of the extinction probability to the invariant graph is pointwise, by definition, but also Hölder.

\begin{corollary}\label{thmCN}
   Let $\alpha\in(0,1]$. Assume (H\ref{hyp3}($\alpha$)). The sequence of functions $(x\mapsto\varphi^{(n)}(x,0))_{n\in\mathbb{N}}$ converges in the $\beta$-Hölder norm to $q$ for all $0<\beta<\alpha$.
   \end{corollary}

Theorem~\ref{thm2.6} and Corollary~\ref{thmCN} are proven in Sub-subsection~\ref{section4.2.2}.

\subsection{Relationship with skew products}\label{section1}

The cocycle relation \begin{align*} 
 \varphi^{(n)}(x,s) = \varphi(x,\varphi^{(n)}(Tx,s)).
    \end{align*} verified by the probability generating function does not allow us to define a skew product transformation. However, assuming that the transformation $T$ is invertible, we can reduce our model to a skew product model. We define $\Psi$, by: 
\fonction{\Psi}{\mathbb{X}\times [0,1]}{\mathbb{X}\times [0,1]}{(x,s)}{(T^{-1}x,\varphi(x,s))}
Then for all $x\in\mathbb{X}$, $n\in\mathbb{N}$, and $s\in[0,1]$, 
\begin{align}
    \Psi^n(x,s)=(T^{-n}x,\varphi^{(n)}(T^{-(n-1)}x,s)).\label{eq11}
\end{align}
Equation~(\ref{eq11}) provides a link with the model of Keller and Otani in \cite{MR3170606}. In their article, they consider $\Theta$ a two-dimensional compact Riemannian manifold, $T:\Theta\mapsto \Theta$ a topologically mixing $\mathcal{C}^2$-Anosov diffeomorphism and $g:\Theta\to(0,\infty)$ a Hölder continuous function. Moreover, for all $t\in\mathbb{R}$, they define a skew product transformation: 
\fonction{T_t}{\Theta\times \mathbb{R}^+}{\Theta\times \mathbb{R}^+}{(\theta,x)}{(T\theta,f_t(\theta,x))}
And they define a fibre function,
\fonctionv{f_t}{\Theta\times \mathbb{R}^+}{ \mathbb{R}^+}{(\theta,x)}{e^{-t}g(\theta)h(x)}
where $h\in\mathcal{C}^1(\mathbb{R}^+,\mathbb{R})$ is strictly concave with $h(0)=0$, $h'(0)=1$, and $\underset{x\to\infty}{\lim}\frac{h(x)}{x}=0$.
For $n\geq 2$, they define:
\fonction{f_t^n}{\Theta\times \mathbb{R}^+}{ \mathbb{R}^+}{(\theta,x)}{f_t(T^{n-1}\theta,f_t^{n-1}(\theta,x))}
For all $t\in\mathbb{R}$, there exists $M_t>0$ such that $f_t(\theta,M_t)<M_t$ for all $\theta\in\Theta$. Moreover, they define for $\theta\in\mathbb{R}$, 
\begin{align*}
    \varphi_t(\theta)\defeq\lim_{n\to\infty}f_t^n(T^{-n}\theta,M_t).
\end{align*}
Finally; they define $N_t=\{\theta\in \Theta:\varphi_t(\theta)=0\}$ for $t\in\mathbb{R}$.

We fix the parameter $t$ in \cite{MR3170606} to provide a link to our article. We can then give the following partial correspondence table: \begin{center}
   \begin{tabular}{|c|c|}
    \hline
    \cite{MR3170606} & this article \\
    \hline
    $\Theta, T^{-1}$ & $\mathbb{X},T$ \\
    $f_t$ & $\varphi$ \\
    $f_t^n$ & $\varphi^{(n)}$ \\
    $N_t$ & $N$\\
    $g_t$& $m$\\
    $\varphi_t$ & $q$\\
    \hline
 \end{tabular} .
\end{center}

The link between $\varphi_t$ and $q$ comes in particular from the fact that they are maximal (respectively minimal) solutions of the respective equations $f_t(\theta,\varphi(\theta))=\varphi(T\theta)$ and $q(x)=\varphi(x,q(Tx))$. This link gives the motivation to Theorem~\ref{thm4}. Our model doesn't need to assume that $T$ is invertible. We consider the dynamics of $T$, contrary to Keller and Otani, who study the dynamics of $T^{-1}$. In our model, see Theorem~\ref{thm5}, we only need an unstable direction.  Whereas the Keller and Otani model has a stable and an unstable direction, only the stable direction of $T$ is important for the dynamic (because the stable direction of $T$ is the unstable direction of $T^{-1}$). 

\begin{example}[Example~\ref{ex1}]\label{ex4}
    The model of Example~\ref{ex1} depends on a parameter $\lambda$, which can be identified with the parameter $-t$ in \cite{MR3170606}. 
We have the same property in both articles: $\log g_t=\log g -t$ and $\log m_\lambda=\log m +\lambda$. In both models, we try to observe the bifurcation of the Hausdorff dimension of $N_t$ (respectively $N_\lambda$) in $t$ (respectively in $\lambda$), see Subsection~\ref{s1}.
\end{example}

\section{Elementary properties}\label{section2}

We prove some elementary results on probability generating functions in the case of the model, on the invariant graph, and on the set of bad environments (defined in Sub-subsection~\ref{section3}).

\subsection{Probability generating functions}

We study the regularity properties of generating functions and the links between generating functions across generations.

\begin{proof}[Proof of Proposition~\ref{prop5}]
We will show that for all $x \in \mathbb{X}$, $n \in \mathbb{N}$, and $s \in [0,1]$,
\begin{align*}
    \varphi^{(n+1)}(x,s) = \varphi^{(n)}(x,\varphi(T^n x,s)).
\end{align*}
Let $x \in \mathbb{X}$, $n \in \mathbb{N}$, and $s \in [0,1]$. Using that $\{Y_{n,k}, k \in \mathbb{N}\}$ is i.i.d.\ and independent of $Z_n(x)$,
\begin{align*}
    \varphi^{(n+1)}(x,s)
    &= \mathbb{E}[s^{Z_{n+1}(x)}] \\
    &= \mathbb{E}[s^{\sum_{k=1}^{Z_n(x)}Y_{n,k}}] \\
    &= \mathbb{E}[\mathbb{E}[s^{\sum_{k=1}^{Z_n(x)}Y_{n,k}}|Z_n(x)]] \\
    &= \mathbb{E}[\mathbb{E}[s^{Y_{n,1}}]^{Z_n(x)}] \\
    &= \mathbb{E}[\varphi(T^n x,s)^{Z_n(x)}] \\
    &= \varphi^{(n)}(x,\varphi(T^n x,s)).
\end{align*}
The proposition follows by induction.
\end{proof}

\begin{corollary}\label{cor5}
    Under (H\ref{hyp1}), for all $n\in\mathbb{N}$, $\varphi^{(n)}$ is continuous on $\mathbb{X}\times[0,1]$.
\end{corollary}

\begin{proof}
Let $x,y\in\mathbb{X}$ and $s,t\in[0,1]$. Then,
\begin{align*}
    |\varphi(x,s)-\varphi(y,t)| &\leq|\varphi(x,s)-\varphi(x,t)|+|\varphi(x,t)-\varphi(y,t)|\\
    &\leq|\varphi(x,s)-\varphi(x,t)|+ \sum_{k=0}^{+\infty}|\mu_x(k)-\mu_y(k)|t^k\\
    &\leq|\varphi(x,s)-\varphi(x,t)|+ \sum_{k=0}^{+\infty}|\mu_x(k)-\mu_y(k)|\text{ since }t\in[0,1]\\
    &=|\varphi(x,s)-\varphi(x,t)|+ \|\mu_x-\mu_y\|_{1}.
\end{align*}
Continuity of $\varphi$ follows from the continuity of $x\in\mathbb{X}\mapsto\mu_x$ and $s\in[0,1]\mapsto\varphi(x,s)$ for all $x\in\mathbb{X}$. 
For $n\in\mathbb{N}$, continuity of $\varphi^{(n)}$ follows by induction and Proposition~\ref{prop5}.
\end{proof}

 For any $x\in\mathbb{X}$, $s\mapsto\varphi(x,s)$ is $\mathcal{C}^\infty$ on $[0,1)$. By monotonicity, its derivative and second derivative are well-defined at one if we admit the value $+\infty$. We denote by $\partial_s\varphi$ (respectively $\partial^2_{s}\varphi$) the derivative (respectively the second derivative) of the function $\varphi$ with respect to the second variable. We use the convention $\log(0)=-\infty$ and $\log(+\infty)=+\infty$.

\begin{lemma}\label{lem3}
    For all $x\in\mathbb{X}$, $n,k\in\mathbb{N}$, and $s\in [0,1]$:
    \begin{align*}
        \log \partial_s \varphi^{(n+k)}(x,s)=\log \partial_s \varphi^{(k)}(T^nx,s)+\log \partial_s \varphi^{(n)}(x,\varphi^{(k)}(T^nx,s)).
    \end{align*}
\end{lemma}
\begin{proof}
    We differentiate Equation (\ref{eq5}) with respect to the second variable.\qedhere
\end{proof}

\begin{corollary}\label{cor1}
    For all $x\in\mathbb{X}$, $s\in[0,1]$, and $n\in\mathbb{N}$, 
    \begin{align}
        \log \partial_s \varphi^{(n)}(x,s)&=\sum_{k=0}^{n-1}\log \partial_s \varphi(T^{k}x,\varphi^{(n-1-k)}(T^{k+1}x,s))\\
        &=\sum_{k=0}^{n-1}\log \partial_s \varphi(T^{n-k-1}x,\varphi^{(k)}(T^{n-k}x,s))\nonumber
        .
    \end{align}
\end{corollary}

\begin{proof}
    We prove the second equality by induction on $n\in\mathbb{N}$.
    For $n=0$, the result holds for all $x\in\mathbb{X}$ and $s\in[0,1]$.\\
    Assume that the result is true at rank $n\in\mathbb{N}$ for all $x\in\mathbb{X}$ and $s\in[0,1]$.
    Let $x\in\mathbb{X}$ and $s\in[0,1]$. 
    \begin{align*}
        \log \partial_s \varphi^{(n+1)}(x,s)
        &=\log\partial_s \varphi^{(n)}(Tx,s)+\log \partial_s \varphi (x,\varphi^{(n)}(Tx,s)) \text{ by Lemma~\ref{lem3},}\\
        &=\sum_{k=0}^{n-1}\log \partial_s \varphi(T^{n-k}x,\varphi^{(k)}(T^{n-k+1}x,s)) +\log\partial_s\varphi(x,\varphi^{(n)}(Tx,s))\\
        &=\sum_{k=0}^{n}\log \partial_s \varphi(T^{n+1-k-1}x,\varphi^{(k)}(T^{n+1-k}x,s)).
    \end{align*}
    This completes the induction. The first equality is obtained immediately by changing the index in the sum.
\end{proof}

Proposition~\ref{prop2} gives a sufficient condition for $\partial_s\varphi$ to be continuous.

\begin{prop}\label{prop2}
Assume (H\ref{hyp1}) and that there exists $\Tilde{\mu}$ a positive measure on $\mathbb{N}$ which has a first moment such that for all $k\in\mathbb{N}$ and $x\in\mathbb{X}$, \begin{align*}
    \mu_x([k,+\infty))\leq \Tilde{\mu}([k,+\infty)).
\end{align*} Then $\partial_s\varphi$ is continuous on $\mathbb{X}\times[0,1]$. In particular, $m$ is continuous on $\mathbb{X}$.
\end{prop}

\begin{proof}
Let $x,y\in\mathbb{X}$ and $s,t\in[0,1]$. Then,
\begin{align*}
    |\partial_s\varphi(x,s)-\partial_s\varphi(y,t)| &\leq|\partial_s\varphi(x,s)-\partial_s\varphi(x,t)|+|\partial_s\varphi(x,t)-\partial_s\varphi(y,t)|\\
    &\leq|\partial_s\varphi(x,s)-\partial_s\varphi(x,t)|+ \sum_{k=0}^{+\infty}k|\mu_x(k)-\mu_y(k)|t^{k-1}\\
    &\leq|\partial_s\varphi(x,s)-\partial_s\varphi(x,t)|+ \sum_{k=0}^{+\infty}k|\mu_x(k)-\mu_y(k)|.
\end{align*}
Moreover,
\begin{align*}
    \sum_{k=0}^{+\infty}k|\mu_x(k)-\mu_y(k)|&=\sum_{k=0}^{+\infty}\sum_{i=k+1}^{+\infty}|\mu_x(i)-\mu_y(i)|\\
    &\leq \sum_{k=0}^{+\infty} \left[\mu_x([k+1,+\infty))+\mu_y([k+1,+\infty))\right]\\
    &\leq 2 \sum_{k=0}^{+\infty} \Tilde{\mu}([k+1,+\infty))\\
    &<+\infty.
\end{align*}
By the dominated convergence theorem, continuity follows from the continuity of $x\in\mathbb{X}\mapsto\mu_x$ and for all $x\in\mathbb{X}$, $s\in[0,1]\mapsto\partial_s\varphi(x,s)$.
\end{proof}

\begin{example}[Example~\ref{ex1}]

We can explicitly compute the probability generating function. Let $\lambda\in\mathbb{R}$.
Then, for all $x\in\RZ$ and $s\in[0,1]$, \begin{align}
    \varphi_\lambda(x,s)&=\exp{ \left(e^{\lambda-\cos{2\pi x}}(s-1)\right) },\nonumber\\
    \partial_s\varphi_\lambda(x,s)&=e^{\lambda-\cos{2\pi x}}\varphi_\lambda(x,s),\label{eq6}
\\
\log m_\lambda(x)&=\lambda-\cos{2\pi x}.\nonumber\end{align} 
The product form of the derivative (\ref{eq6}) is specific to Poisson distributions and provides a property similar to \cite{MR3170606} (see Example~\ref{ex4}).
\end{example}

\subsection{The invariant graph}

We study the probability of extinction $q$ and justify the term invariant graph. Lemma~\ref{lem2} is a stronger version of Proposition~\ref{prop1}.

\begin{lemma}\label{lem2}
Assume (H\ref{hyp1}).
\begin{itemize}
    \item $q(x)=\varphi^{(k)}(x,q(T^kx))$ for all $x\in\mathbb{X}$ and $k\in\mathbb{N}$,
    \item $q$ is lower semi-continuous.
\end{itemize}
\end{lemma}

\begin{proof}~
\begin{itemize}
    \item Let $x\in\mathbb{X}$ and $k\in\mathbb{N}$.
    \begin{align*}
        q(x)&=\lim_{n\to\infty}\varphi^{(n+k)}(x,0)\\
        &=\lim_{n\to\infty}\varphi^{(k)}(x,\varphi^{(n)}(T^kx,0))\text{ by Corollary~\ref{prop5}} \\
        &=\varphi^{(k)}(x,\lim_{n\to\infty}\varphi^{(n)}(T^kx,0)) \text{ by continuity of }\varphi^{(k)} \\
        &=\varphi^{(k)}(x,q(T^kx)).
    \end{align*}
    \item Since $q(x)=\underset{n\to\infty}{\lim}\nearrow \varphi^{(n)}(x,0)=\underset{n\in\mathbb{N}}{\sup}~\varphi^{(n)}(x,0)$, $q$ is lower semi-continuous as a supremum of continuous functions.\qedhere
\end{itemize}
\end{proof}

We measure the convergence speed towards the invariant graph with the function $F$.

\begin{definition}\label{def1}
  We define \( F \) for all \( x\in\mathbb{X} \) by:
\begin{align*}
    F(x)\defeq&\log \partial_s\varphi(x,q(Tx)).
\end{align*}
\end{definition}

Lemma~\ref{lem7} measures the convergence speed to the invariant graph as a Birkhoff sum.
\begin{lemma}\label{lem7}
Assume (H\ref{hyp1}), for all \( x\in\mathbb{X} \) and \( n\in\mathbb{N} \):
    \begin{align}\label{eq4}
        \log \partial_s \varphi^{(n)}(x,q(T^nx))=\sum_{k=0}^{n-1}F(T^kx).
    \end{align}
\end{lemma}

\begin{proof}
    Let \( x\in\mathbb{X} \) and \( n\in\mathbb{N} \). Apply Corollary~\ref{cor1} with \( s=q(T^nx) \). We obtain:
    \begin{align*}
     \log \partial_s \varphi^{(n)}(x,q(T^nx))&=\sum_{k=0}^{n-1}\log \partial_s \varphi(T^{k}x,\varphi^{(n-1-k)}(T^{k+1}x,q(T^nx)))\\
     &=\sum_{k=0}^{n-1}\log \partial_s \varphi(T^{k}x,q(T^{k+1}x))\text{ by Lemma~\ref{lem2}}.\qedhere
    \end{align*}
\end{proof}

\begin{example}[Example~\ref{ex1}] By Lemma~\ref{lem2}, $q_\lambda$ is the smallest solution of the functional equation: \begin{align*}
    q_\lambda(x)=\exp{\big(e^{\lambda-\cos{2\pi x}}(q_\lambda(Tx)-1)}\big) \text{ for all }x\in\RZ.
\end{align*}
    \begin{figure}[!ht]
    \center
\includegraphics[scale=0.85]{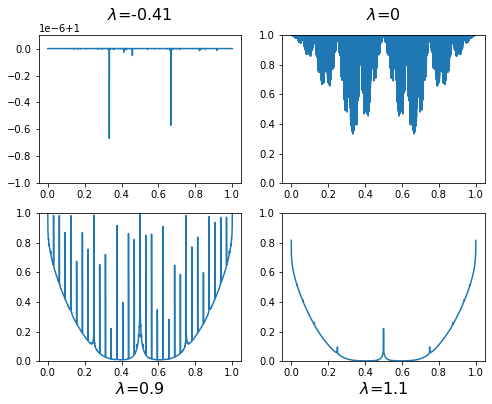}
\captionof{figure}{Plot of the probability of extinction $q_\lambda$ for some $\lambda\in\mathbb{R}$.}
\label{figure2}
\end{figure}
\begin{remark}
    The scale of the first graph is different from the others.
\end{remark}

\end{example}

\subsection{The set of bad environments}

We study the set of bad environments, a $T$ invariant and $G_\delta$ set.

\begin{prop}\label{cor2}
Assume (H\ref{hyp1}). Then $N=T^{-1}N$. In particular if $\nu\in\mathcal{E}_T(\mathbb{X})$, then $\nu(N)\in\{0,1\}$.
\end{prop}

\begin{proof} Let $\nu\in\mathcal{E}_T(\mathbb{X})$.
\begin{align*} N&=\{x\in\mathbb{X}:q(x)=1\}.\\
    &=\{x\in\mathbb{X}:\varphi(x,q(Tx))=1\} \text{ by Lemma~\ref{lem2}.}\\
    &=\{x\in\mathbb{X}:q(Tx)=1\} \text{ because for all }x\in\mathbb{X},\,\mu_x\ne\delta_0\\
    &=T^{-1}N.
\end{align*}
As $\nu\in\mathcal{E}_T(\mathbb{X})$ and $N$ is $T$-invariant, $\nu(N)\in\{0,1\}$.
\end{proof}

Lower semi-continuity of $q$ allows us to show that $N$ is a $G_\delta$ set.

\begin{lemma}\label{lm5}
Assume (H\ref{hyp1}). Then $N$ is a $G_\delta$ set.
\end{lemma}

\begin{proof}
For any $k\in\mathbb{N}$, the set ${\{x\in\mathbb{X}: q(x)>1-\frac{1}{k}\}}$ is open since $q$ is lower semi-continuous. Moreover,
\begin{align*}
  N&=\{x\in\mathbb{X}:q(x)=1\} \\
  &= \bigcap_{k\in\mathbb{N}}{\left\{x\in\mathbb{X}:q(x)>1-\frac{1}{k}\right\}} \text{ since }q\leq 1.\qedhere
\end{align*}
\end{proof}

\begin{example}[Example~\ref{ex1}]\label{ex2}

Assume (H\ref{hyp1}) and that $T$ is a expanding map of $\RZ$, then $N$ is either empty or $N$ is a dense $G_\delta$ set.
Indeed, if $N$ is non-empty, then there exists $x\in\RZ$ such that $x\in N_\lambda$.
By Proposition~\ref{cor2}, $T^{-1}N\subset N$, so \begin{align*}
     \underset{n\in\mathbb{N}}{\bigcup}T^{-n}(x)\subset\underset{n\in\mathbb{N}}{\bigcup}T^{-n}N\subset N.
\end{align*}
And $\underset{n\in\mathbb{N}}{\bigcup}T^{-n}(x)$ is dense in $\RZ$ as $T$ is expanding.

In particular, in the case of Example~\ref{ex1}, for all $\lambda\in\mathbb{R}$, $N_\lambda$ is either empty or $N_\lambda$ is a dense $G_\delta$ set.
\end{example}

\section{The set of bad environments}

In this section, we study the metric and topological properties of the set of bad environments.

\subsection{Measure of the set of bad environments}\label{section3.1}

It is sometimes difficult to say anything about the process $(Z_n(x))_{n\in\mathbb{N}}$ for all $x\in\mathbb{X}$. We can then use the point of view of ergodic theory by choosing $x$ according to an invariant (and even ergodic without losing generality) measure. In this case, the dynamic environment $(\mu_{T^nx})_{n\in\mathbb{N}}$ is a stationary and ergodic process that allows us to use Athreya and Karlin's results \cite{MR0298780}. They obtained a criterion to determine whether the measure of the set $N$ is zero or full when considering an ergodic measure.

We consider $\nu\in\mathcal{E}_T(\mathbb{X})$. As seen in Proposition~\ref{cor2}, $\nu(N)\in\{0,1\}$. We must determine whether $\nu(N)$ is 0 or 1. Theorem~\ref{thmAK}, Corollary~\ref{cor4}, and Theorem~\ref{thm2} are three results of Athreya and Karlin presented in the context of our model.

\begin{theorem}\cite[Theorem 1]{MR0298780}\label{thmAK}
If $\nu(N)=0$ and $\mathbb{E}_\nu[\log m(x)]^+<\infty$, then:
\begin{itemize}
    \item $\mathbb{E}_\nu|\log m(x)|<\infty$ and $\mathbb{E}_\nu[\log m(x)]>0$.
    \item $\mathbb{E}_\nu\left|\log\frac{1-q(x)}{1-q(Tx)}\right|<\infty$ and $\mathbb{E}_\nu\left[\log\frac{1-q(x)}{1-q(Tx)}\right]=0$.
\end{itemize}
\end{theorem}

By contrapositive, we deduce Corollary~\ref{cor4}:

\begin{corollary}\label{cor4}\cite[Corollary 1]{MR0298780}
If $\mathbb{E}_\nu[\log m(x)]^+<\infty$ and \\ $\mathbb{E}_\nu[\log m(x)]^+\leq \mathbb{E}_\nu[\log m(x)]^-\leq\infty$, then $\nu(N)=1$.
\end{corollary}

Moreover, we have the following reciprocal:

\begin{theorem}\label{thm2}\cite[Theorem 3]{MR0298780}
If $\mathbb{E}_\nu[-\log(1-\varphi(x,0))]<\infty$ and $\mathbb{E}_\nu[\log m(x)]^-<\mathbb{E}_\nu[\log m(x)]^+\leq\infty$, then $\nu(N)=0$.
\end{theorem}

\begin{proof}[Proof of Theorem~\ref{thm1}]
   Corollary~\ref{cor4} and Theorem~\ref{thm2} (from the results of Athreya and Karlin \cite{MR0298780}) provide criteria to determine whether or not there is almost sure extinction in the case where $\nu$ is an ergodic measure. Specifically, $\nu(N)=0$ if and only if $\mathbb{E}_\nu[\log m(x)]>0$.
Indeed, the two integrability hypotheses are verified here:
\begin{itemize}
    \item $\mathbb{E}_\nu[-\log(1-\varphi(x,0))]<\infty$ because for all $x\in\mathbb{X}$, $\varphi(x,0)=\mu_x(0)<1$ and $\mu$ is continuous on the compact $\mathbb{X}$.
    \item $\mathbb{E}_\nu[\log m(x)]^-<\infty$ or $\mathbb{E}_\nu[\log m(x)]^+<\infty$ because for all $x\in\mathbb{X}$, $m(x)\geq1-\mu_x(0)>0$ and $\mu$ is continuous on the compact $\mathbb{X}$.\qedhere
\end{itemize} 
\end{proof}

\subsection{Dimension of the set of bad environments}\label{section2.2}

In Theorem~\ref{thm4} and Theorem~\ref{thm5}, we give the Hausdorff dimension of the set of bad environments in the critical case on two types of systems.
\begin{definition}For each $x\in\mathbb{X}$,
    \begin{align*}
        \Gamma(x)\defeq \liminf_{n\in\mathbb{N^*}}\frac{1}{n} \log \partial_s \varphi^{(n)}(x,1)=\liminf_{n\in\mathbb{N^*}}\frac{1}{n}\sum_{k=0}^{n-1}\log m(T^kx).
    \end{align*}
\end{definition}

In Lemma~\ref{lem6}, we observe that $\Gamma(x)$ is the right quantity to know if $x$ is in the set of bad environments.

\begin{lemma}\label{lem6} Assume (H\ref{hyp1}) and that $C\defeq\underset{x\in\mathbb{X}}{\sup}~\underset{k=0}{\overset{+\infty}{\sum}}k^2\mu_x(k)<+\infty$. For each $x\in\mathbb{X}$: 
    \begin{itemize}
        \item if $x\in N$, then $\Gamma(x)\leq0$,
        \item if $x\notin N$, then $\Gamma(x)\geq0$.
    \end{itemize}
\end{lemma}

\begin{proof}~
\begin{itemize}
\item Let $x\in\mathbb{X}$ such that $\Gamma(x)>0$. As for any $n\in\mathbb{N}$, $m(T^nx)\leq C$ and $\partial^2_{s}\varphi(T^nx,1)\leq C$. Since $\mu$ is continuous, and $\mu_{y}({0})<1$ for every $y\in\mathbb{X}$ (which is compact), there exists $c>0$ such that for every $n\in\mathbb{N}$, $m(T^nx)>c$. Since $\Gamma(x)>0$, there exists $\delta>0$ and $N\in\mathbb{N}$ such that for every $n\geq N$, $\partial_s\varphi^{(n)}(x,1)\geq e^{\delta n}$.

Thus, the sequence given by: 
\begin{align*}
   \frac{1}{\partial_s\varphi^{(n-1)}(x,1)} \left(\frac{\partial^2_{s}\varphi(T^nx,1)}{ m(T^nx)^2}+\frac{1-m(T^nx)}{m(T^nx)} \right)
\end{align*}
is the general term of a convergent series. Hence, the assumption \cite[A]{MR4094390} and the condition \cite[vii Theorem 1]{MR4094390} are satisfied, so $q(x)<1$, i.e.\ $x\notin N$.

\item Let $x\in\mathbb{X}$ such that $\Gamma(x)<0$. There exists an extracted sequence $(n_k)_{k\in\mathbb{N}}$ and $\delta>0$ such that for all $k\in\mathbb{N}$, \begin{align*}
    \partial_s \varphi^{(n_k)}(x,1)\leq e^{-\delta n_k}.
\end{align*}
Then for each $k\in\mathbb{N}$, 
\begin{align*}
    q(x)&\geq \varphi^{(n_k)}(x,0)\\
    &=1-\int_0^1 \partial_s \varphi^{(n_k)}(x,t) \,\mathrm{d}t\\
    &\geq 1 - \sup_{t\in [0,1]} \partial_s \varphi^{(n_k)}(x,t)\\
    &\geq 1- e^{-\delta n_k}\underset{k\to +\infty}{\longrightarrow} 1.
\end{align*}
Thus, $q(x)=1$, i.e.\ $x\in N$.\qedhere
\end{itemize}    
\end{proof}

By Proposition~\ref{prop2}, under the assumptions of Lemma~\ref{lem14}, $m$ is continuous. But, if $\mu$ is Hölder continuous, we need a stronger integrability hypothesis that $m$ inherits the Hölder regularity of $\mu$.

\begin{lemma}\label{lem14}
    Assume (H\ref{hyp1}), that there exists $\varepsilon>0$ such that $\underset{x\in\mathbb{X}}{\sup}~\underset{k=0}{\overset{+\infty}{\sum}}k^{1+\varepsilon}\mu_x(k)<+\infty$, and that $x\in\mathbb{X}\mapsto \mu_x$ is $\alpha$-Hölder continuous for the $\ell^1$ norm (for $\alpha\in (0,1]$). Then $m$ is $\alpha\frac{\varepsilon}{1+\varepsilon}$ Hölder continuous.
\end{lemma}

\begin{proof}
    Let $x,y\in\mathbb{X}$, by Hölder's inequality,
    \begin{align*}
        |m(x)-m(y)|&\leq\sum_{k=0}^{+\infty}k|\mu_x(k)-\mu_y(k)|\\
        &\leq\left(\sum_{k=0}^{+\infty}k^{1+\varepsilon}|\mu_x(k)-\mu_y(k)|\right)^{1/(1+\varepsilon)}\left(\sum_{k=0}^{+\infty}|\mu_x(k)-\mu_y(k)|\right)^{\varepsilon/(1+\varepsilon)}\\
        &\leq C^{1/(1+\varepsilon)}\lVert\mu_x-\mu_y\rVert_1^{\varepsilon/(1+\varepsilon)}\\
        &\leq C^{1/(1+\varepsilon)} |\mu|_\alpha^{\varepsilon/(1+\varepsilon)}d(x,y)^{\alpha\varepsilon/(1+\varepsilon)}.\qedhere
    \end{align*}
\end{proof}

\begin{proof}[Proof of Theorem~\ref{thm4}]
  In the case where $\mathbb{X}$ is a two-dimensional compact Riemannian manifold, the proof of Theorem~\ref{thm4} is the same as the proof of \cite[Theorem 2]{MR3170606} (which uses results from \cite[Section 10 and 12]{MR2434246}) using the correspondence table from Section \ref{section1}, Lemma~\ref{lem6} and the Hölder continuity of $m$ (Lemma~\ref{lem14}). The results of \cite{MR2434246} were originally proved in \cite{MR1837214,MR1994883,MR2238885}. However, we study the dynamic of $T$, unlike Keller and Otani, who study the dynamics of $T^{-1}$. We will therefore use $u\defeq\log \lVert dT|E^u\rVert$ and for all $\lambda\in (\lambda_{\min},\lambda_{\max})$, \begin{align*}
     D (\lambda)= \max\left\{\frac{h_T(\nu)}{\nu(\log \lVert dT|E^u\rVert)}:\nu\in\mathcal{P}_T(\mathbb{X})\text{ and } \int_{\mathbb{X}}\log m(x)\, \mathrm{d}\nu(x)= \lambda\right\}.
\end{align*} 

If $T$ is a $\mathcal{C}^2$ uniformly expanding transformation of $\RZ$, the system is conjugated to a one-sided full shift of finite type (using a Markov partition). The entropy function is upper semi-continuous, as $T$ is expansive. Thus, the proofs of \cite[Theorem 2]{MR3170606} and \cite{MR2434246} can be adapted by considering only the dilating direction.
We will therefore use for all $\lambda\in (\lambda_{\min},\lambda_{\max})$, \begin{align*}
     D (\lambda)= \max\left\{\frac{h_T(\nu)}{\nu(\log T')}:\nu\in\mathcal{P}_T(\RZ)\text{ and } \int_{\mathbb{R}/\mathbb{Z}}\log m(x)\, \mathrm{d}\nu(x)= \lambda\right\}.
\end{align*} 
The maximum of the function $D$ is reached by: \begin{align*}
    \lambda=\int_{\mathbb{R}/\mathbb{Z}}\log m(x)\, \mathrm{d}\nu_{Leb}(x)
\end{align*} and is equal to one where $\nu_{Leb}$ is the unique invariant probability measure absolutely continuous with respect to the Lebesgue measure (see \cite{MR0245761}).
To adapt \cite[Lemma 2]{MR3170606}, unlike Keller and Otani, removing a dimension is unnecessary because our model has only an unstable dimension.
Results of \cite[Chapter 10]{MR2434246} (10.1.4, 10.3.1, and 10.1.6) and \cite[Example 7.2.5]{MR2434246} are applicable to our model. This gives us Lemma~\ref{lem15} and Lemma~\ref{lem16}.
For $\lambda\in (\lambda_{\min},\lambda_{\max})$, let:
\begin{align*}
    S_\lambda=\left\{ x\in\RZ:\frac{1}{n}\sum_{i=0}^{n-1}\log m(T^ix)\underset{n\to+\infty}{\longrightarrow}\lambda\right\}.
\end{align*}

\begin{lemma}\cite[Lemma 2]{MR3170606}\label{lem15}
    For $\lambda\in (\lambda_{\min},\lambda_{\max})$, $dim_H(S_\lambda)=D(\lambda)$.
\end{lemma}
Moreover, the function $D$ is analytic, but we only need the continuity of $D$ to prove Theorem~\ref{thm4}.
\begin{lemma}\cite[Lemma 3]{MR3170606}\label{lem16}
    The function $\lambda\in(\lambda_{\min},\lambda_{\max})\mapsto D(\lambda)$
is continuous.\end{lemma}
Lemma~\ref{lem15} and Lemma~\ref{lem16} allow us to prove Theorem~\ref{thm4}. By Lemma~\ref{lem6}, if $\lambda\in(\lambda_{\min},0)$, then $S_\lambda\subset N$. Let $\lambda_c\defeq\mathbb{E}_{\nu_{Leb}}[\log m]$.
\begin{itemize}
 \item  Case $\lambda_c\leq 0$. Let $\lambda\in(\lambda_{\min},\lambda_c)$. So, \begin{align*}
      dim_H(S_\lambda)\leq dim_H(N)\leq 1.
 \end{align*}
 Therefore, by Lemma~\ref{lem15} and Lemma~\ref{lem16},\begin{align*}
     dim_H(S_\lambda)=D(\lambda)\underset{\lambda\to\lambda_c^-}{\longrightarrow}D(\lambda_c)=1.
 \end{align*}
 Thus, \begin{align*}
     dim_H(N)= 1.
 \end{align*}
 \item Case $\lambda_c> 0$. Let $\lambda\in(\lambda_{\min},0)$. So,
 \begin{align*}
      dim_H(S_\lambda)\leq dim_H(N),
 \end{align*} Therefore, by Lemma~\ref{lem15} and Lemma~\ref{lem16},\begin{align*}
     dim_H(S_\lambda)=D(\lambda)\underset{\lambda\to0^-}{\longrightarrow}D(0).
 \end{align*}
 Let,
\begin{align*}
    N^+=\left\{ x\in\RZ:\liminf_{n\to +\infty}\frac{1}{n}\sum_{i=0}^{n-1}\log m(T^ix)\leq 0\right\}.
\end{align*}
By Lemma~\ref{lem6}, $N\subset N^+$. There remains to bound the dimension of $N^+$ by $D(0)$ using the lower pointwise dimension.
\begin{lemma}\cite[Lemma 4]{MR3170606}\label{lem17} There exists $\nu\in\mathcal{P}(\RZ)$ such that
  for all $x\in N^+$, \begin{align*}
     \underline{d}_{\nu}(x)\defeq\liminf_{r \to 0}\frac{\log \nu\left(\mathcal{B}(x,r)\right)}{\log r}\leq D(0). 
  \end{align*}  
\end{lemma}

\begin{proof}
Let $p\geq 2$ be the degree of $T$. Consider a Markov partition of $\RZ$, and let $\sigma$ be the associated one-sided topological Markov Chain defined on $\{1,\ldots,p\}^\mathbb{N}$. We also consider the coding map $\chi:\{1,\ldots,p\}^\mathbb{N}\to \RZ$ obtained from the Markov partition.
Define the following functions on $\{1,\ldots,p\}^\mathbb{N}$:\begin{align*}
    \phi&\defeq \log m \circ \chi,\\
   d&\defeq \log T' \circ \chi.
\end{align*}
Let $D\defeq D(0)$ and $a(w)\defeq q\phi(w)-D d(w)$ where $q\in\mathbb{R}$. As $T$ is topologically mixing, there exists a unique $q\in\mathbb{R}$ and a unique equilibrium measure $\Tilde{\nu}$ of $a$ such that $P_\sigma(a)=0$ and $\int_{\{1,\ldots,p\}^\mathbb{N}}\phi(w)\,\mathrm{d}\Tilde{\nu}(w)=0$. Moreover, $\Tilde{\nu}$ is a Gibbs measure and $q<0$ (see \cite[Lemma 12.3.3]{MR2434246} and the proof of \cite[Lemma 3]{MR3170606}).

For all $n\in\mathbb{N}$, for all $(i_0,\ldots,i_n)\in\{1,\ldots,p\}^{n+1}$, we define the cylinder: \begin{align*}
    C_{i_0,\ldots,i_n}\defeq \{i_0\}\times \ldots\times \{i_n\}\times \{1,\ldots,p\}^\mathbb{N}.
\end{align*}
As $\Tilde{\nu}$ is a Gibbs measure and $P(a)=0$, there exists a constant $K>0$ such that $n\in\mathbb{N}$, for all $(i_0,\ldots,i_n)\in\{1,\ldots,p\})^{n+1}$, and for all $\omega\in C_{i_0,\ldots,i_n}$:
\begin{align}\label{e2}
    \frac{1}{K}<\frac{\Tilde{\nu}(C_{i_0,\ldots,i_n})}{\exp\sum_{i=0}^{n-1}a(\sigma^i \omega)}<K.
\end{align}

Let $\nu\defeq\Tilde{\nu}\circ \chi^{-1}$ a measure on $\RZ$ defined on the Markov partition and\begin{align*}
    \Omega_0\defeq \left\{w\in \{1,\ldots,p\}^\mathbb{N}, \liminf_{n\to\infty}\sum_{n=0}^{+\infty}\phi(\sigma^n w)\leq 0\right\}.
\end{align*}
Let $\delta>0$, $w\in\Omega_0$, and $x=\chi(\omega)\in\RZ$, there exists an increasing sequence $(n_k)_{k\in\mathbb{N}}\in \mathbb{N}^\mathbb{N}$ such that for all $k\in\mathbb{N}$, \begin{align}\label{e1}
    q\sum_{i=0}^{n_k}\phi(\sigma^i w)\geq q n_k \delta.
\end{align}
Let $(r_k)_{k\in\mathbb{N}}\in\mathbb{R}^N$ such that for all $k\in\mathbb{N}$, \begin{alignat}{3}\label{e3}
    -\sum_{i=0}^{n_k} d(\sigma ^i \omega)>\log r_k &\text{ and } &-\sum_{i=0}^{n_k+1} d(\sigma ^i \omega)\leq \log r_k.
\end{alignat}
As $d>0$ and continuous, \begin{alignat}{3}\label{e4}
    -(n_k+1)\inf d> \log r_k &\text{ and } &  -(n_k+2)\sup d\leq\log r_k.
\end{alignat}

In particular, as $n_k\underset{k\to\infty}{\longrightarrow}+\infty $, $r_k\underset{k\to\infty}{\longrightarrow}0 $. Let $(i_n)_{n\in\mathbb{N}}\in\{1\ldots p\}^\mathbb{N}$ be the sequence such that for all $n\in\mathbb{N}$, $\omega\in C_{i_0,\ldots,i_n}$.
Combining (\ref{e1}) and (\ref{e2}) we obtain, there exists $\rho>1$ such that for all $k\in\mathbb{N}$, \begin{align*}
   \nu( \mathcal{B}(x,\rho r_k))&\geq \Tilde{\nu}(C_{i_0,\ldots,i_{n_k}})\\
   & \geq \frac{1}{K}\exp \left(q \sum_{i=0}^{n_k}\phi(\sigma ^i \omega)-D\sum_{i=0}^{n_k}d(\sigma ^i \omega)\right)\\
   & \geq \frac{1}{K}\exp \left(q n_k \delta -D\sum_{i=0}^{n_k}d(\sigma ^i \omega)\right)
\end{align*}
So by (\ref{e3}) and by (\ref{e4}):
\begin{align*}
    \nu( \mathcal{B}(x,\rho r_k))\geq \frac{1}{K}r_k^D\exp\left( q\delta \left(\frac{-\log r_k}{\sup d}-2\right)\right).
\end{align*}
Thus, \begin{align*}
    \liminf_{r\to 0}\frac{\log \nu( \mathcal{B}(x, r))}{\log r}\leq D - \frac{q \delta }{\sup d}. 
\end{align*}
As the result is true for all $\delta>0$, \begin{align*}
    \liminf_{r\to 0}\frac{\log \nu( \mathcal{B}(x, r))}{\log r}\leq D.
\end{align*}
Moreover, $N^+= \chi(\Omega_0)$, which concludes the proof of the lemma.
\end{proof}

So, by \cite[Theorem 2.1.5]{MR2434246}, the bound of the lower pointwise dimension obtained in Lemma~\ref{lem17} gives the dimension of Hausdorff of $N$:
    \begin{align*}
D(0)\leq \dim_H(N)\leq&\dim_H(N^+)\leq D(0).\qedhere
    \end{align*}
\end{itemize}
\end{proof}

\subsection{Application to the example of the doubling map}\label{s1}

We examine the implications of Theorem~\ref{thm1} and Theorem~\ref{thm4} in the case of Example~\ref{ex1}.

\paragraph{Measure of the set of bad environments}~\\
 Let $\nu\in\mathcal{E}_T(\RZ)$ and $\lambda\in\mathbb{R}$. Then, \begin{align*}
    \mathbb{E}_{\nu}[\log m_\lambda(x)]=&\lambda-\mathbb{E}_\nu[\cos 2\pi x].
\end{align*}
Therefore, by Theorem~\ref{thm1}, \begin{align}\label{e17}
    \nu(N_\lambda)=1 \text{ if and only if }\lambda \leq \mathbb{E}_\nu[\cos 2\pi x].
\end{align}

As $\RZ$ is a compact set, $\inf$ and $\sup$ are reached. Then for all $\lambda\in\mathbb{R}$, if:\begin{itemize}
    \item $\lambda\leq \lambda_{\min}$, then for all $\nu\in\mathcal{P}_T(\RZ)$, $\nu(N_\lambda)=1$, 
    \item $\lambda_{\min}<\lambda\leq \lambda_{\max}$, there exits $\nu_0,\nu_1\in\mathcal{P}_T(\RZ)$ such that $\nu_0(N_\lambda)=1$ and $\nu_1(N_\lambda)=0$,
    \item $\lambda>\lambda_{\max}$, then for all $\nu\in\mathcal{P}_T(\RZ)$, $\nu(N_\lambda)=0$.
\end{itemize}

We determine the values of $\lambda_{\min}$ and $\lambda_{\max}$. For any $\nu\in\mathcal{P}_T(\RZ)$,  $\mathbb{E}_\nu[\cos 2\pi x]\in[-1,1] $. $\delta_0\in\mathcal{P}_T(\RZ)$ and $\mathbb{E}_{\delta_0}[\cos 2\pi x]=1$, then $\lambda_{\max}=1$. By \cite{MR1785392}, the atomic ergodic measure $\nu_0=\frac{1}{2}\delta_{\frac{1}{3}}+\frac{1}{2}\delta_{\frac{2}{3}}$ is the invariant measure that minimizes the quantity $\mathbb{E}_\nu[\cos 2\pi x]$ (see Figure \ref{figure2}, first graph) and 
$\mathbb{E}_{\nu_0}[\cos 2\pi x]=\frac{1}{2}\cos\frac{2\pi}{3}+\frac{1}{2}\cos\frac{4\pi}{3}=-\frac{1}{2}$. Thus $\lambda_{\min}=-\frac{1}{2}.$

If $\lambda>1=\lambda_{\max}$, then for any $x\in\RZ$, the laws $\text{Pois}(e^{\lambda-\cos2\pi x})$ stochastically dominate the law $\text{Pois}(e^{\lambda-1})$, which has an expectation $\lambda-1$ strictly greater than 1. Thus, by coupling, we can conclude that $N_\lambda=\emptyset$.

If $\lambda\leq 1=\lambda_{\max}$, then there exists $\nu\in\mathcal{E}_T(\RZ)$ such that $\nu(N_\lambda)=1$. In particular, $N_\lambda$ is non-empty, so by Example~\ref{ex2}, $N_\lambda$ is a dense $G_\delta$.
Lemma~\ref{lem6} allow us that for $\lambda<\lambda_{\min}=-\frac{1}{2}$, $N_\lambda=\RZ$. Indeed, let $\lambda<\lambda_{\min}=-\frac{1}{2}$. $\log m_\lambda$ is continuous, and for all $\nu\in\mathcal{P}_T(\RZ)$, \begin{align*}
        \mathbb{E}_{\nu}[\log m_{\lambda}(x)]\leq \lambda_{\min}-\lambda<0.
    \end{align*}
    By the semi-uniform ergodic theorem \cite[Theorem 1.9]{MR1734626}, there exists $\varepsilon>0$ and $N\in\mathbb{N}^*$ such that for all $x\in\RZ$ and $n\geq N$, 
    \begin{align*}
        \frac{1}{n}\sum_{k=0}^{n-1}\log m_\lambda(T^kx)<\varepsilon.
    \end{align*} By the contrapositive of Lemma~\ref{lem6}, $x\in N_\lambda$ for all $x\in\RZ$.
We summarize the different regimes of this example on the following graduated line:\newline
\begin{tikzpicture}[every text node part/.style={align=center},scale=1.16]
\draw[ -][dotted] (-0.2, 0) -- (4, 0) node[above][black]{$\lambda_{\min}$};
\draw[ -][dashed] (4, 0) -- (10, 0) node[above][black]{$\lambda_{\max}$};
\draw[ -] (10, 0) -- (12.25, 0) node[above][black]{$\lambda$};
\draw[ ->] (12.24, 0) -- (12.25, 0) ;
\foreach \x/\y in {0/-1.5,2/-1,4/-0.5,6/0,8/0.5,10/1,12/1.5} \draw (\x,-.10)node[below]{$\y$} -- (\x,0.10);
\node[draw][dotted] at (1.8,-1.1){$N_\lambda=\RZ$};
\node[draw][dashed] at (7,-1.2){$\text{There exists}~\nu_0,\nu_1\in\mathcal{E}_T(\RZ) \text{ such that }$\\ $\nu_0(N_\lambda)=0 \text{ and }\nu_1(N_\lambda)=1$};
\node[draw] at (11.2,-1.1){$N_\lambda=\emptyset$};
\draw[dashed] (10,0) circle (0.09);
\end{tikzpicture}

For $\lambda=\lambda_{min}=-\frac{1}{2}$, by Theorem~\ref{thm1}, $\nu(N_\lambda)=1$ for all $\nu\in\mathcal{P}_T(\RZ)$ but we don't know if $N_\lambda=\RZ$.

More generally, \cite{MR1785392} allows us to study $\underset{\nu\in\mathcal{P}_T(\rz)}{\sup}\mathbb{E}_\nu[\cos 2\pi (x-\omega)]$ for $\omega\in\RZ$, and therefore to study the different regimes if we take as the law of reproduction $\mu_{\lambda,w,.}\defeq x\in\RZ\mapsto\text{Pois}\left(e^{\lambda-\cos{(2\pi (x-\omega))}}\right)$ for $\omega\in\RZ$ and $\lambda\in\mathbb{R}$. The maximum of $\mathbb{E}_\nu[\cos 2\pi (x-\omega)]$ is reached on Sturm measure and is periodic for all $\omega\in\RZ$ outside a certain set of Hausdorff dimension zero.

\paragraph{Dimension of the set of bad environments}~\\
The Lebesgue measure on $\RZ$ is ergodic. Thus, Theorem~\ref{thm4} implies: \begin{align*}
        dim_H(N_\lambda)=
    \left\{\begin{array}{ll}
    &1\text{ if } \lambda\in \left(-\frac{1}{2},0\right], \\
    &\underset{\nu\in\mathcal{P}_T(\mathbb{X})}{\max}\left\{\frac{h_T(\nu)}{\log 2}:\int_{\mathbb{R}/\mathbb{Z}} \cos (2 \pi x)\, \mathrm{d}\nu(x)=\lambda\right\} \text{ if } \lambda\in[0,1).
    \end{array}
    \right.
    \end{align*}
    Moreover, $\lambda\in\left(-\frac{1}{2},1\right)\mapsto dim_H(N_\lambda)$ is continuous (see Lemma~\ref{lem16}) and non-increasing.  The graph of $dim_H(N_\lambda)$ as a function of $\lambda$ can be computed using the methods developed in \cite{PhysRevE.87.042913}.

    For $p\in(0,1)$, let $m(p)$ be the Bernoulli product of parameter $p$. By \cite{MR2334791}, $dim_H(m(p))=-\big(p\log_2(p)+(1-p)\log_2(1-p)\big)$. If $\lambda\in\left(0,1\right)$, for all parameters $p\in(0,1)$ close enough to $\frac{1}{2}$, the probability of survival of the process $(Z_n(x))_{n\in\mathbb{N}}$ is positive for $m(p)$-almost all $x\in\mathbb{X}$.

\section{Regularity of the invariant graph}

In this section, we study the continuity (Theorem~\ref{thm3}) and Hölder continuity (Theorem~\ref{thm2.6}) of the invariant graph in the supercritical case. 

\subsection{Continuity in the supercritical case}\label{section4.1}

We will now examine the uniformly supercritical case, i.e., for all $\nu\in\mathcal{E}_T(\mathbb{X})$,\newline $\mathbb{E}_\nu[\log m(x)] >0$. With some additional assumptions, the function $q$ is continuous. We already know that $q$ is lower semi-continuous by Lemma~\ref{lem2}. Now, we will express $q$ as an infimum of continuous functions to make $q$ upper semi-continuous and hence continuous. We start by proving that the function $q$ is upper-bounded by a constant strictly less than 1.

\begin{lemma}\label{lem8}
Assume (H\ref{hyp2}). There exists $K<1$ and $N\in\mathbb{N}$ such that for all $x\in\mathbb{X}$, the sequence $(\varphi^{(nN)}(x,K))_{n\in\mathbb{N}}$ is decreasing and converges to $q(x)$.
\end{lemma}

\begin{proof}
For all $\nu\in\mathcal{P}_T(\mathbb{X})$, 
\begin{align*}
\int_{\mathbb{X}} \log m(x) \,\mathrm{d}\nu(x)>0.
\end{align*}
Thus, there exists $\varepsilon>0$ such that for all $\nu\in\mathcal{P}_T(\mathbb{X})$, 
\begin{align*}
\int_{\mathbb{X}} \log m(x) \,\mathrm{d}\nu(x)>\varepsilon
\end{align*}
because of the compactness of $\mathcal{P}_T(\mathbb{X})$ (since $\mathbb{X}$ is compact), and because $m$ is continuous by Proposition~\ref{prop2}. 
By the semi-uniform ergodic theorem \cite[Theorem 1.9]{MR1734626}, there exists $N\in\mathbb{N}$ such that for all $n\geq N$ and $x\in\mathbb{X}$,
\begin{align*}
    \frac{1}{n}\sum_{k=0}^{n-1}\log m(T^k x)>\frac{\varepsilon}{2}.
\end{align*}
In particular, for all $x\in\mathbb{X}$,
\begin{align*}
    \partial_s\varphi^{(N)}(x,1)>\exp\left(\frac{\varepsilon N}{2}\right)>1.
\end{align*}
By induction and Proposition~\ref{prop2},  $\partial_s\varphi^{(N)}$ is uniformly continuous on the compact $\mathbb{X}\times [0,1]$.
Thus, there exists $K<1$ such that for all $x\in\mathbb{X}$ and $t\in[K,1]$, 
\begin{align*}
    \partial_s\varphi^{(N)}(x,t)>1.
\end{align*}
Let $x\in\mathbb{X}$, $\varphi^{(N)}(x,1)=1$, so $\varphi^{(N)}(x,K)\leq K$.
Thus, for all $n\in\mathbb{N}$,
\begin{align*}
    \varphi^{((n+1)N)}(x,K)&=\varphi^{(Nn)}(x,\varphi^{(N)}(T^{nN}x,K)) \text{ by Proposition~\ref{prop5}}\\
    &\leq \varphi^{(Nn)}(x,K)   
\end{align*}
because $s\in[0,1]\mapsto\varphi^{(nN)}(x,s)$ is increasing and $\varphi(T^{Nn}x,K)\leq K$.
 
Thus, the sequence $(\varphi^{(nN)}(x,K))_{n\in\mathbb{N}}$ is decreasing. Moreover, for all $n\in\mathbb{N}$,\newline $\varphi^{(nN)}(x,0)\leq\varphi^{(nN)}(x,K)$, so $q(x)\leq K$.

For $n\in\mathbb{N}$, $\varphi^{(nN)}(x,0)\leq\varphi^{(Nn)}(x,K)$ because $s\in[0,1]\mapsto\varphi^{(nN)}(x,s)$ is increasing. Since $\varphi^{(nN)}(x,0)$ converges to $q(x)$, we only need to show that $\varphi^{(nN)}(x,K)-\varphi^{(nN)}(x,0)\to0$. 

We will prove this using \cite[Theorem 5]{MR0298780}. This theorem states that given a sequence of reproduction laws, the probability generating function associated with the population size at the $n$th generation converges to a function $g$ for all $s\in[0,1]$. Moreover, 
since there exists $c<1$ such that for all $y\in\mathbb{X}$, $\mu_y(1)<c$ (because $\mu$ is continuous, and for all $y\in\mathbb{X}$, $\mu_y(1)<1$), then for all $s\in [0,1)$, $g(s)=g(0)$. Therefore, $\varphi^{(nN)}(x,K)-\varphi^{(nN)}(x,0)\underset{n\to +\infty}{\longrightarrow}0$.
\end{proof}
Lemma~\ref{lem8} allows us to prove Theorem~\ref{thm3}.
\begin{proof}[Proof of Theorem~\ref{thm3}]
By Lemma~\ref{lem2}, $q$ is lower semi-continuous. According to \newline Lemma~\ref{lem8}, $q(x)=\underset{n\in\mathbb{N}}{\inf}\varphi^{(Nn)}(x,K)$, and for all $n\in\mathbb{N}$, $x\in\mathbb{X}\mapsto\varphi^{(n)}(x,K)$ is continuous. Thus, $q$ is upper semi-continuous, and consequently, $q$ is continuous.
\end{proof}

By Lemma~\ref{lem8}, for all $a\in[0,1)$, $(\varphi^{(n)}(.,a))_{n\in\mathbb{N}}$ converges pointwise to $q$. Corollary~\ref{cor6} shows that the convergence is uniform.

\begin{corollary}\label{cor6}
    Assume (H\ref{hyp2}). For each $a\in[0,1)$, the sequence of functions \newline $(\varphi^{(n)}(.,a))_{n\in\mathbb{N}}$ converges uniformly to $q$.
\end{corollary}

\begin{proof}
We define a function $\Phi$ as:
\fonction{\Phi}{\mathcal{C}(\mathbb{X},[0,1])}{\mathcal{C}(\mathbb{X},[0,1])}{f}{\left(x\in\mathbb{X}\mapsto\varphi(x,f(Tx))\right)}
 $\Phi$ is uniformly continuous. Indeed, let $\varepsilon>0$. The function $(x,s)\in\mathbb{X}\times[0,1]\mapsto\varphi(x,s)$ is uniformly continuous. Hence, there exists $\eta>0$ such that for all $x\in\mathbb{X}$ and $s,t\in[0,1]$ with $|s-t|\leq \eta$, we have $|\varphi(x,s)-\varphi(x,t)|\leq\varepsilon$. Let $f,g\in\mathcal{C}(\mathbb{X},[0,1])$ such that $\lVert f-g \rVert_{\infty}\leq \eta$. Then, for all $x\in\mathbb{X}$, $
|f(Tx)-g(Tx)|\leq \eta$. 
Thus, for all $x\in\mathbb{X}$, $|\varphi(x,f(Tx))-\varphi(x,g(Tx))|\leq \varepsilon$, i.e.\ $\lVert \Phi(f)-\Phi(g) \rVert_{\infty}\leq \varepsilon$. Therefore, $\Phi$ is uniformly continuous.

Let $a\in [0,1)$, $0<K<1$ defined as in Lemma~\ref{lem8}, and $\Tilde{K}\in[K,1)$ such that $\Tilde{K}>a$. The sequence of continuous and increasing functions $(\varphi^{(n)}(.,0))_{n\in\mathbb{N}}$ (respectively the sequence of continuous and decreasing functions $(\varphi^{(nN)}(.,\Tilde{K}))_{n\in\mathbb{N}}$) converges pointwise to the continuous (by Theorem~\ref{thm3}) function $q$. Moreover, since $\mathbb{X}$ is compact, by Dini's theorem, the sequences $(\varphi^{(n)}(.,0))_{n\in\mathbb{N}}$ and $(\varphi^{(nN)}(.,\Tilde{K}))_{n\in\mathbb{N}}$ converge uniformly to $q$.
The uniform convergence of $(\varphi^{(nN)}(.,\Tilde{K}))_{n\in\mathbb{N}}$ to $q$ can be expressed as:
\begin{align*}  
    \Phi^{nN}(\Tilde{\mathbb{K}})\underset{n\to +\infty}{\longrightarrow} q
\end{align*}
(where $\Tilde{\mathbb{K}}$ is the constant function equal to $\Tilde{K}$).
For every $k\in\llbracket0,N-1\rrbracket$, by the continuity of $\Phi$ and by Lemma~\ref{lem2}:
\begin{align*}  
    \Phi^{nN+k}(\Tilde{\mathbb{K}})\underset{n\to +\infty}{\longrightarrow} \Phi^k(q)=q .
\end{align*}
Thus, $(\varphi^{(n)}(.,\Tilde{K}))_{n\in\mathbb{N}}$ converges uniformly to $q$.
By squeezing, $(\varphi^{(n)}(.,a))_{n\in\mathbb{N}}$ converge uniformly to $q$.
\end{proof}

\subsection{Hölder regularity in the supercritical case}\label{section4.2}

We will now see in the case where the invariant graph is continuous whether the Hölder regularity of the reproduction function $\mu$ is preserved by the invariant graph $q$.

Assume (H\ref{hyp2}) and that $q(x)>0$ for all $x\in\mathbb{X}$. By Proposition~\ref{prop2}, $m$ is continuous, and by Theorem~\ref{thm3}, $q$ is continuous. Denote $\beta\defeq\underset{x\in\mathbb{X}}{\sup}\, m(x)< +\infty$, $K\defeq\underset{x\in\mathbb{X}}{\sup}\,q(x)<1$ and $\kappa\defeq\underset{x\in\mathbb{X}}{\inf}\,q(x)>0$.

\subsubsection{Lyapunov exponent in the fibre}

In this section, we will study the Lyapunov exponent in the fibre \begin{align*}
    \lambda_F=\lim_{n\to\infty}\frac{1}{n}\sup_{x\in\mathbb{X}}\log(\partial_s\varphi^{(n)}(x,q(T^n x))).
\end{align*} We start by showing in Lemma~\ref{lem9} and Lemma~\ref{lem10} that $\lambda_F$ is well defined:

\begin{lemma}\label{lem9}
   Assume (H\ref{hyp2}) and that $q(x)>0$ for all $x\in\mathbb{X}$. Then $x\in\mathbb{X}\mapsto F(x)=\log \partial_s\varphi(x,q(Tx))$ takes finite values and is continuous. 
\end{lemma}

\begin{proof}
    $q$ is continuous, and for all $x\in\mathbb{X}$, $0<\kappa<q(x)<K<1$. Then by Proposition~\ref{prop2}, $x\in\mathbb{X}\mapsto \partial_s\varphi(x,q(Tx))$ is continuous. All we have to do is check that $\partial_s\varphi(x,q(Tx))>0$ for all $x\in\mathbb{X}$. Let $x\in\mathbb{X}$.
    \begin{align*}
         \partial_s \varphi(x,q(Tx))
        &= \sum_{k=1}^{+\infty} k \mu_x(k) q(Tx)^{k-1}\\
        &\geq \sum_{k=1}^{+\infty } \mu_x(k) \kappa^{k-1}\\
        &>0 \text{ because } \mu_x(0)<1.\qedhere
    \end{align*}
\end{proof}

Lemma \ref{lem10} proves that $\lambda_F$ is well defined using Fekete's subadditive lemma.

\begin{lemma}\label{lem10}
    Assume (H\ref{hyp2}) and that $q(x)>0$ for all $x\in\mathbb{X}$. Then\newline $\big(\frac{1}{n}\underset{x\in\mathbb{X}}{\sup}\log(\partial_s\varphi^{(n)}(x,q(T^n x)))\big)_{n\in\mathbb{N}}$ has a finite limit.
\end{lemma}
\begin{proof}
Let $x\in\mathbb{X}$ and $n,m\in\mathbb{N}$, by Lemma~\ref{lem7}:
\begin{align*}
    \frac{1}{n}\log(\partial_s\varphi^{(n)}(x,q(T^n x)))=\frac{1}{n}\sum_{k=0}^{n-1}F(T^kx).
\end{align*}
And,
\begin{align*}
    \sup_{x\in\mathbb{X}}\sum_{k=0}^{n+m-1}F(T^kx)&\leq \sup_{x\in\mathbb{X}}\sum_{k=0}^{n-1}F(T^kx)+\sup_{x\in\mathbb{X}}\sum_{k=n}^{n+m-1}F(T^kx)\\
    &\leq \sup_{x\in\mathbb{X}}\sum_{k=0}^{n-1}F(T^kx)+\sup_{x\in\mathbb{X}}\sum_{k=0}^{m-1}F(T^kx).
\end{align*}
Then $\left(\underset{x\in\mathbb{X}}{\sup}\underset{k=0}{\overset{n-1}{\sum}}F(T^kx)\right)_{n\in\mathbb{N}}$ is subadditive, so by Fekete's subadditive lemma,\newline $\frac{1}{n}\underset{x\in\mathbb{X}}{\sup}\log(\partial_s\varphi^{(n)}(x,q(T^n x)))$ has a limit in $\{-\infty\}\cup\mathbb{R}$.
Moreover, $F$ is continuous on $\mathbb{X}$ (by Lemma~\ref{lem9}) so the limit is finite.
\end{proof}

Proposition~\ref{prop4} and Lemma~\ref{lem11} give other equivalent definitions of the Lyapunov exponent in the fibre.

\begin{prop}\label{prop4}
    Assume (H\ref{hyp2}) and that $q(x)>0$ for all $x\in\mathbb{X}$. Then for all $0<a<1$,
    \begin{align*}
        \underset{n\to\infty}{\lim} \underset{x\in\mathbb{X}}{\sup}\frac{1}{n}\log \partial_s \varphi^{(n)}(x,a)=\lambda_F.
    \end{align*}
\end{prop}

\begin{proof}
Let $\varepsilon>0$ and $\delta\defeq\min(\frac{\kappa}{2},a)$. The function $\log \partial_s\varphi$ is uniformly continuous on the compact $\mathbb{X}\times\left[\delta,1\right]$. Thus, there exists $\eta>0$ such that for all $x\in\mathbb{X}$ and $s,t\in\left[\delta,1\right]$ with $|s-t|\leq \eta$, $|\log \partial_s\varphi(x,s)-\log \partial_s\varphi(x,t)|\leq \varepsilon$.\\
By Corollary~\ref{cor6}, there exists $m\in\mathbb{N}$ such that $\underset{x\in\mathbb{X}}{\sup}|q(x)-\varphi^{(m+k)}(x,a)|<\min(\eta,\delta)$ for all $k\in\mathbb{N}$.\\
Let $n\in\mathbb{N}$ and $x\in\mathbb{X}$. Using Lemma~\ref{lem3} and Lemma~\ref{lem2},
\begin{align}
    \log \partial_s \varphi^{(n+m)}(x,a)
    &=\log \partial_s \varphi^{(m)}(T^nx,a)+\log \partial_s \varphi^{(n)}(x,\varphi^{(m)}(T^nx,a))\label{e5}
\end{align}
And,
\begin{align}
\log \partial_s \varphi^{(n+m)}&(x,q(T^{n+m}x))\nonumber\\
=&\log \partial_s \varphi^{(m)}(T^nx,q(T^{n+m}x))+\log \partial_s \varphi^{(n)}(x,\varphi^{(m)}(T^nx,q(T^{n+m}x)))\nonumber\\
=&\log \partial_s \varphi^{(m)}(T^nx,q(T^{n+m}x))+\log \partial_s \varphi^{(n)}(x,q(T^{n}x)).\label{e6}
\end{align}
There exists $C>0$ such that the continuous function $(x,s)\in\mathbb{X}\times\left[\delta,1\right]\mapsto\log \partial_s \varphi^{(m)}(x,s)$ is bounded by $C$.

Using Corollary~\ref{cor1}, Lemma~\ref{lem2}, and by Proposition~\ref{prop5},
\begin{align}
    \Big|\log \partial_s \varphi^{(n)}(x,\varphi^{(m)}&(T^nx,a))-\log \partial_s \varphi^{(n)}(x,q(T^{n}x))\Big|\nonumber\\
&= \sum_{k=0}^{n-1}\Big|\log \partial_s\varphi(T^{k}x,\varphi^{(n-1-k)}(T^{k+1}x,\varphi^{(m)}(T^nx,a)))\nonumber\\
&\hspace{4em}-\log \partial_s\varphi(T^{k}x,\varphi^{(n-1-k)}(T^{k+1}x,q(T^{n}x)))\Big|\nonumber\\
    &= \sum_{k=0}^{n-1}\Big|\log \partial_s\varphi(T^{k}x,\varphi^{(n+m-1-k)}(T^{k+1}x,a))\nonumber\\ 
    &\hspace{4em}-\log \partial_s\varphi(T^{k}x,q(T^{k+1}x))\Big|\nonumber\\
    &\leq n\varepsilon \text{ since } | \varphi^{(n+m-1-k)}(T^{k+1}x,a)-q(T^{k+1}x)|\leq \eta\label{e7}\\ &\hspace{4em} \text{ and } \varphi^{(n+m-1-k)}(T^{k+1}x,a),q(T^{k+1}x)\in [\delta,1].\nonumber
\end{align}
Hence, by Equality~\ref{e5}, Equality~\ref{e6}, and Inequality~\ref{e7},
\begin{align}
    \Big|\limsup_{n\to +\infty}&\frac{1}{n}\sup_{x\in\mathbb{X}} \log \partial_s \varphi^{(n)}(x,a)-\lambda_F\Big|\nonumber\\
    &=\Big|\limsup_{n\to +\infty}\frac{1}{n}\sup_{x\in\mathbb{X}} \log \partial_s \varphi^{(n)}(x,a)-\lim_{n\to +\infty}\frac{1}{n}\sup_{x\in\mathbb{X}} \log \partial_s \varphi^{(n)}(x,q(T^nx)\Big|\nonumber\\
    &\leq\limsup_{n\to +\infty}\frac{1}{n+m}\Big|\sup_{x\in\mathbb{X}} \log \partial_s\varphi^{(n+m)}(x,a)-\sup_{x\in\mathbb{X}}\log \partial_s \varphi^{(n+m)}(x,q(T^{n+m}x))\Big|\nonumber\\
&=\limsup_{n\to +\infty}\frac{1}{n+m}\Big|\sup_{x\in\mathbb{X}}\big(\log \partial_s \varphi^{(m)}(T^nx,a)+\log \partial_s \varphi^{(n)}(x,\varphi^{(m)}(T^nx,a))\big)\nonumber\\
&\hspace{2cm}-\sup_{x\in\mathbb{X}}\big(\log \partial_s \varphi^{(m)}(T^nx,q(T^{n+m}x))+\log \partial_s \varphi^{(n)}(x,q(T^{n}x))\big)\Big|\nonumber\\
&\leq \limsup_{n\to +\infty}\frac{1}{n+m}\Big(\sup_{x\in\mathbb{X}}\big|\log \partial_s \varphi^{(n)}(x,\varphi^{(m)}(T^nx,a))-\log \partial_s \varphi^{(n)}(x,q(T^{n}x))\big|\nonumber\\
&\hspace{2cm}+\sup_{x\in\mathbb{X}}\big|\log \partial_s \varphi^{(m)}(T^nx,a)\big|+\sup_{x\in\mathbb{X}}\big|\log \partial_s \varphi^{(m)}(T^nx,q(T^{n+m}x))\big|\Big)\nonumber\\
&\leq \limsup_{n\to +\infty}\frac{1}{n+m}(n\varepsilon+2C)=\varepsilon\label{e8}
\end{align}
Likewise,
\begin{align}
    \Big|\liminf_{n\to +\infty}&\frac{1}{n}\sup_{x\in\mathbb{X}} \log \partial_s \varphi^{(n)}(x,a)-\lambda_F\Big|\leq \varepsilon\label{e9}
\end{align}
The conclusion of the proposition follows from Inequality~\ref{e8} and Inequality~\ref{e9}.
\end{proof}

Lemma~\ref{lem11} allows us to express $\lambda_F$ without using limits.

\begin{lemma}\label{lem11}
    Assume (H\ref{hyp2}) and that $q(x)>0$ for all $x\in\mathbb{X}$. Then \begin{align*}
        \lambda_F= \sup_{\nu\in\mathcal{P}_T(\mathbb{X})}\int_{\mathbb{X}} F(x) \,\mathrm{d}\nu(x).
    \end{align*}
\end{lemma}

\begin{proof}
    The function $F$ is continuous (by Lemma~\ref{lem9}). By semi-uniform ergodic theorem \cite[Theorem 1.9]{MR1734626} and by Equation (\ref{eq4}), we have:
\begin{align}\label{eq1}
\lambda_F\leq\sup_{\nu\in\mathcal{P}_T(\mathbb{X})}\int_{\mathbb{X}} F(x) \,\mathrm{d}\nu(x).
\end{align}
Since $\mathcal{P}_T(\mathbb{X})$ is compact, the supremum in (\ref{eq1}) is reached by an ergodic measure that we denote $\nu_0\in\mathcal{E}_T(\mathbb{X})$. By Birkhoff's ergodic theorem applied to the continuous function $F$ and by Equation (\ref{eq4}), for $\nu_0$-almost every $x\in\mathbb{X}$,
\begin{align*}
\lim_{n\to\infty}\frac{1}{n}\log \partial_s \varphi^{(n)}(x,q(T^nx))=\int_{\mathbb{X}} F(x) \,\mathrm{d}\nu_0(x).
\end{align*}

Thus, for $\nu_0$-almost every $x\in\mathbb{X}$, 
\begin{align*}
    &\lambda_F\geq \lim_{n\to\infty}\frac{1}{n}\log \partial_s \varphi^{(n)}(x,q(T^nx))=\sup_{\nu\in\mathcal{P}_T(\mathbb{X})}\int_{\mathbb{X}} F(x) \,\mathrm{d}\nu(x).\qedhere
\end{align*}
\end{proof}

Proposition~\ref{prop6} shows that $\lambda_F$ controls the exponential convergence to the invariant graph.

\begin{prop}\label{prop6}
   Assume (H\ref{hyp2}) and that $q(x)>0$ for all $x\in\mathbb{X}$. Then for all $a\in[0,1[$,
   \begin{align*}
       \limsup_{n\to\infty}\frac{\log\lVert \varphi^{(n)}(.,a)-q \rVert_{\infty}}{n}\leq \lambda_F
   \end{align*}
\end{prop}

\begin{proof}
    Let $a\in [0,1)$, $\Tilde{K}\defeq\max(a,K)\in(0,1)$, $x\in\mathbb{X}$, and $n\in\mathbb{N}$. Using that $t\in[0,1]\mapsto\partial_s\varphi^{(n)}(x,t)$ and $t\in[0,1]\mapsto\varphi^{(n)}(x,t)$ are non-decreasing,
    \begin{align*}
        |\varphi^{(n)}(x,a)-q(x)|\leq \varphi^{(n)}(x,\Tilde{K})-\varphi^{(n)}(x,0)\leq \partial_s\varphi^{(n)}(x,\Tilde{K})\Tilde{K}
    \end{align*}
    Therefore, by Proposition~\ref{prop4}, 
    \begin{align*}
       &\limsup_{n\to\infty}\frac{\log\lVert \varphi^{(n)}(.,a)-q \rVert_{\infty}}{n}\leq \limsup_{n\to\infty}\frac{1}{n}\underset{x\in\mathbb{X}}{\sup}\log \partial_s\varphi^{(n)}(x,\Tilde{K})=\lambda_F\qedhere
   \end{align*}
\end{proof}

Lemma~\ref{lem12} shows that convergence towards the invariant graph is exponential.

\begin{lemma}\label{lem12}
Assume (H\ref{hyp2}) and that $q(x)>0$ for all $x\in\mathbb{X}$. Then $\lambda_F< 0$.
\end{lemma}

\begin{proof}
    Let $x\in\mathbb{X}$. Using Lemma~\ref{lem2}, Jensen's inequality, and that for all $x\in \mathbb{X}$, $t\in[0,1]\mapsto\partial_s\varphi(x,t)$ is non-decreasing:
    \begin{align}
        \log (1-q(x))&=\log(1-\varphi(x,q(Tx))\nonumber\\
        &=\log\left(\int_{q(Tx)}^1\partial_s\varphi(x,t) \,\mathrm{d}t\right)\nonumber\\
        &=\log(1-q(Tx))+\log\left(\int_{q(Tx)}^1\partial_s\varphi(x,t) \,\frac{\mathrm{d}t}{1-q(Tx)}\right)\nonumber\\
        &\geq\log(1-q(Tx))+\int_{q(Tx)}^1\log(\partial_s\varphi(x,t)) \,\frac{\mathrm{d}t}{1-q(Tx)}\nonumber\\
        &\geq\log(1-q(Tx))+\log(\partial_s\varphi(x,q(Tx)))\label{eq10}.
    \end{align}
    Let $\nu\in\mathcal{P}_T(\mathbb{X})$. As $\mathbb{E}_\nu[\log m(x)]>0$, there exists $A\in\mathcal{B}(\mathbb{X})$ such that for all $x\in A$, $\mu_x(\{0,1\})<1$ and $\nu(A)>0$. Thus, for all $x\in A$, $t\mapsto\partial_s\varphi(x,t)$ is increasing and Inequation (\ref{eq10}) is strict. We integrate Inequation (\ref{eq10}) according to $\nu$. Using that $\nu$ is $T$-invariant,
    \begin{align*}
\int_{\mathbb{X}}\partial_s\varphi(x,q(Tx))\,\mathrm{d}\nu(x)<0.
    \end{align*}
Thus, by Lemma~\ref{lem11}, $\lambda_F<0$ because $\mathcal{P}_T(\mathbb{X})$ is compact.
\end{proof}

\subsubsection{Proof of the Hölder regularity of the invariant graph}\label{section4.2.2}

We can now proceed to the proof of Theorem~\ref{thm2.6}. Lemma~\ref{lem13} shows that $\varphi$ inherits the Hölder regularity from $\mu$.

\begin{lemma}\label{lem13}
Let $\alpha\in(0,1]$, assume (H\ref{hyp3}($\alpha$)). Then
$\varphi$ is $\alpha$-Hölder continuous of seminorm smaller than $|\mu|_\alpha$ in the first variable uniformly in the second variable. In other words, for all $s\in[0,1]$, $|\varphi(.,s)|_\alpha\leq |\mu|_\alpha $.
\end{lemma}

\begin{proof}
    For all $s\in[0,1]$ and $x,y\in\mathbb{X}$,
    \begin{align*}
        |\varphi(x,s)-\varphi(y,s)|&\leq \sum_{k=0}^{+\infty} s^k |\mu_x(k)-\mu_y(k)|\\
        &\leq \lVert \mu_x -\mu_y \rVert_1 \text{ because } s\in[0,1],\\
        &\leq |\mu|_\alpha d(x,y)^\alpha \text{ because $\mu$ is $\alpha$-Hölder continuous.}\qedhere
    \end{align*}
\end{proof}

Then Lemma~\ref{lem2.4} shows that for all $n\in\mathbb{N}$, $\varphi^{(n)}$ inherits the Hölder regularity from $\varphi$.

\begin{lemma}\label{lem2.4}
     Let $\alpha\in(0,1]$, assume (H\ref{hyp3}($\alpha$)). For all $n\in\mathbb{N}$, the function $\varphi^{(n)}$ is $\alpha$-Hölder continuous in the first variable uniformly in the second variable. In other words, there exists $C_n>0$ such that for all $s\in[0,1]$, $|\varphi^{(n)}(.,s)|_\alpha\leq C_n$.
\end{lemma}

\begin{proof}
    We prove this result by induction on $n\in\mathbb{N}$. For $n=0$, the result is true.
    
   Assuming the result is true for $n\in\mathbb{N}$. Let $s\in[0,1]$ and $x,y\in\mathbb{X}$. Using Lemma~\ref{lem13} and the induction hypothesis,
    \begin{align*}
        |\varphi^{(n+1)}(x,s)-\varphi^{(n+1)}(y,s)|
       &= |\varphi(x,\varphi^{(n)}(Tx,s))-\varphi(y,\varphi^{(n)}(Ty,s))|\\
        &\leq  |\varphi(x,\varphi^{(n)}(Tx,s))-\varphi(y,\varphi^{(n)}(Tx,s))|\\
        & \hspace{4em}+ |\varphi(y,\varphi^{(n)}(Tx,s))-\varphi(y,\varphi^{(n)}(Ty,s))|\\
        &\leq  |\mu|_\alpha d(x,y)^\alpha + \beta |\varphi^{(n)}(Tx,s)-\varphi^{(n)}(Ty,s)|\\
        &\leq  |\mu|_\alpha d(x,y)^\alpha + \beta C_n d(Tx,Ty)^\alpha\\
        &\leq  |\mu|_\alpha d(x,y)^\alpha + \beta C_n \lVert T\rVert_{lip}^\alpha d(x,y)^\alpha\\
        &\leq  C_{n+1} d(x,y)^\alpha \text{ where } C_{n+1}=|\mu|_\alpha+\beta C_n \lVert T\rVert_{lip}^\alpha.
    \end{align*}
    Thus, $\varphi^{(n+1)}$ is $\alpha$-Hölder continuous (with seminorm $C_{n+1}$) in the first variable uniformly in the second.
\end{proof}

Controlling the seminorm of $\alpha$-Hölder continuity of $\varphi^{(n)}(.,s)$ allows us to obtain, by taking the limit, that the function $q$ is also $\alpha$-Hölder continuous.

\begin{proof}[Proof of Theorem~\ref{thm2.6}]
    Let $\varepsilon>0$ such that $\lambda_F+\varepsilon+\alpha(\lambda_u+\varepsilon)<0$. Such an $\varepsilon$ exists since $\lambda_u\leq 0$ or $ \left(\lambda_u>0 \text{ and }\alpha<-\frac{\lambda_F}{\lambda_u}\right)$.
    There exists $N_0\in\mathbb{N}$ such that for all $n\geq N_0$, $\frac{1}{n}\log(\lVert T^n\rVert_{Lip})<\lambda_u+\varepsilon$. There exists $N_1\in\mathbb{N}$ such that for all $n\geq N_1$ and $s\in[0,K]$, $\underset{x\in\mathbb{X}}{\sup}\frac{1}{n}\log(\partial_s\varphi^{(n)}(x,s))<\lambda_F+\varepsilon$ by Proposition~\ref{prop4}. Let $N=\max(N_0,N_1)$, $A_N=e^{N(\lambda_F+\varepsilon+\alpha(\lambda_u+\varepsilon))}$, and $C_N$ defined as in Lemma~\ref{lem2.4}. We have $A_N<1$. Start by proving Lemma~\ref{lem2.3}, which controls the composition of continuous Hölder functions:
    
    \begin{lemma}\label{lem2.3}
        Let $f\in\mathcal{C}(\mathbb{X},[0,K])$. Suppose that $f$ is $\alpha$-Hölder continuous. Then $x\in\mathbb{X}\mapsto\varphi^{(N)}(x,f(T^Nx))$ is $\alpha$-Hölder continuous, and $|\varphi^{(N)}(.,f(T^N.))|_\alpha\leq |f|_\alpha A_N+C_N$.
    \end{lemma}

    \begin{proof}
        Let $x,y\in\mathbb{X}$.
        \begin{align*}
            |\varphi^{(N)}(x,f(T^Nx))-\varphi^{(N)}(y,f(T^Ny))|&\leq|\varphi^{(N)}(x,f(T^Nx))-\varphi^{(N)}(y,f(T^Nx))|\\
            & \hspace{4em}+|\varphi^{(N)}(y,f(T^Nx))-\varphi^{(N)}(y,f(T^Ny))|\\
            &\leq C_Nd(x,y)^\alpha + e^{N(\lambda_F+\varepsilon)}|f(T^Nx)-f(T^Ny))|\\
            &\leq C_Nd(x,y)^\alpha+e^{N(\lambda_F+\varepsilon)}|f|_\alpha d(T^Nx,T^Ny)^\alpha\\
            &\leq C_Nd(x,y)^\alpha+e^{N(\lambda_F+\varepsilon)}|f|_\alpha e^{N\alpha(\lambda_u+\varepsilon)}d(x,y)^\alpha.
        \end{align*}
        Thus, $x\in\mathbb{X}\mapsto\varphi^{(N)}(x,f(T^Nx))$ is $\alpha$-Hölder continuous, and
        \begin{align*}
            \varphi^{(N)}(.,f(T^N.))|_\alpha &\leq |f|_\alpha A_N+C_N.\qedhere
        \end{align*}
    \end{proof}

    Lemma~\ref{lem2.3} implies Corollary~\ref{cor2.2}, allowing us to control the seminorm $\alpha$-Hölder of $\varphi^{(nN)}$. Corollary~\ref{cor2.2} and Lemma~\ref{lem2.2} allow us to check the Hölder regularity of q.

    \begin{corollary}\label{cor2.2}
        For all $n\in\mathbb{N}^*$, $x\in\mathbb{X}\mapsto \varphi^{(nN)}(x,0)$ is $\alpha$-Hölder continuous with a seminorm less than $\frac{C_N}{1-A_N}$.
    \end{corollary}

    \begin{proof}
        Prove this result by induction on $n\in\mathbb{N}^*$. For $n=1$, the result is true by Lemma~\ref{lem2.4} and since $C_N\leq\frac{C_N}{1-A_N}$.
        
        Assuming the result is true at rank $n\in\mathbb{N}^*$. Then for all $x\in\mathbb{X}$, $\varphi^{(N(n+1))}(x,0)=\varphi^{(N)}(x,\varphi^{(nN)}(T^Nx,0))$. We can apply Lemma~\ref{lem2.3} with the function $x\mapsto\varphi^{(Nn)}(x,0)$, which is $\alpha$-Hölder continuous by the induction hypothesis. Thus, $x\mapsto\varphi^{(N(n+1))}(x,0)$ is $\alpha$-Hölder continuous. Moreover,
        \begin{align*}
            |\varphi^{(N(n+1))}(.,0)|_\alpha &\leq A_N  |\varphi^{(nN)}(.,0)|_\alpha +C_N \text{ by Lemma~\ref{lem2.3}}\\
            &\leq A_N \frac{C_N}{1-A_N}+C_N \text{ by induction hypothesis}.\\ 
            &= \frac{C_N}{1-A_N}.
        \end{align*}
        Therefore, $x\in\mathbb{X}\mapsto \varphi^{((n+1)N)}(x,0)$ is $\alpha$-Hölder continuous with a seminorm less than $\frac{C_N}{1-A_N}$.
    \end{proof}

    Now, we can return to the proof of Theorem~\ref{thm2.6}. By definition, for all $x\in\mathbb{X}$, $q(x)=\underset{n\to\infty}{\lim}\varphi^{(n)}(x,0)$. Moreover, by Corollary~\ref{cor2.2}, for all $n\in\mathbb{N}^*$, $x\mapsto\varphi^{(nN)}(x,0)$ is $\alpha$-Hölder continuous, and $|\varphi^{(nN)}(.,0)|_\alpha\leq \frac{C_N}{1-A_N}$. Thus, by Lemma~\ref{lem2.2}, $q$ is $\alpha$-Hölder continuous.\qedhere
\end{proof} 

Now that we have proven that the function $q$ is $\alpha$-Hölder continuous, we will show that the convergence of the sequence of functions $(x\mapsto\varphi(x,0))_{n\in\mathbb{N}}$ to $q$ is not only pointwise (as defined by default) but also a Hölder convergence.

\begin{proof}[Proof of Corollary~\ref{thmCN}]
    Let $0<\beta<\alpha$, $t=\frac{\beta}{\alpha}\in(0,1)$, and $N\in\mathbb{N}$, $A_N<1$ and $C_N\geq 0$ defined as in Theorem~\ref{thm2.6} and Lemma~\ref{lem2.4}.
    By Corollary~\ref{cor2.2}, for all $n\in\mathbb{N}$, $|\varphi^{(nN)}(.,0)|_\alpha\leq \frac{C_N}{1-A_N}$ and by the proof of Lemma~\ref{lem2.4}, $|\varphi^{(n+1)}(.,0)|_\alpha\leq |\mu|_\alpha+\beta|\varphi^{(n)}(.,0)|_\alpha \lVert T\rVert_{lip}^\alpha$. In particular, $(|\varphi^{(n)}(.,0)|_\alpha)_{n\in\mathbb{N}}$ is bounded by a constant denoted $C$.
    Moreover, by Proposition~\ref{prop2.3}, for all $n\in\mathbb{N}$:
    \begin{align*}
        |q-\varphi^{(n)}(.,0)|_\beta&\leq (2\lVert q-\varphi^{(n)}(.,0)\rVert_{\infty})^{1-t} |q-\varphi^{(n)}(.,0)|_\alpha^{t}\\
        &\leq (2\lVert q-\varphi^{(n)}(.,0)\rVert_{\infty})^{1-t} (|q|_\alpha+|\varphi^{(n)}(.,0)|_\alpha)^{t}\\
        &\leq (2\lVert q-\varphi^{(n)}(.,0)\rVert_{\infty})^{1-t} (|q|_\alpha+C)^{t}.
    \end{align*}
    Since $(\varphi^{(n)}(.,0))_{n\in\mathbb{N}}$ converges uniformly to $q$ by Corollary~\ref{cor6}, it follows that \newline $|q-\varphi^{(n)}(.,0)|_\beta$ converges to 0 as $n$ tends to $+\infty$. Thus, $(\varphi^{(n)}(.,0))_{n\in\mathbb{N}}$ converges in the $\beta$-Hölder norm to $q$.
\end{proof}

\subsubsection{Application to the example of the doubling map}

We examine the regularity of the invariant graph in the case of Example~\ref{ex1}. By Theorem~\ref{thm3}, for all $\lambda>1$, $q_\lambda$ is continuous. Moreover, by Theorem~\ref{thm2.6}, for all $\lambda>1$, $q_\lambda$ is $\alpha$-Hölder continuous for all $\alpha\in(0,1]$, such that $\alpha<-\frac{\lambda_F}{\log{2}}$ (because $\lambda_u=\log{2})$.

\begin{figure}[!ht]
    \center
\includegraphics[scale=0.46]{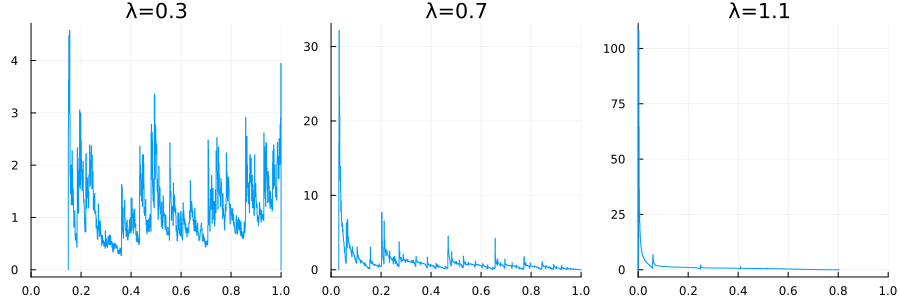}
\captionof{figure}{Histogram of the push-forward of the Lebesgue measure by $q_\lambda$ for some $\lambda\in\mathbb{R}$.}
\label{figure3}
\end{figure}

 The Lebesgue measure is ergodic for the transformation $T:x\mapsto 2x$, and according to \eqref{e17}, $Leb(N_\lambda)=0$ if and only if $\lambda>0$.

In Figure~\ref{figure3}, for $\lambda=0.3$ and $\lambda=0.7$ (critical case), the push-forward of the Lebesgue measure by $q_\lambda$ is lower bounded by a positive constant. Indeed, $\mu_{\lambda,x}(0)>0$ (the probability of having zero offspring according to the law $\mu_{\lambda,x}$ is positive) and continuous in $x\in\RZ$. 

 For $\lambda=1.1$ (uniformly supercritical case), $q_\lambda$ is upper bound by a constant strictly smaller than 1 and $\alpha$-Hölder continuous for some parameter $\alpha\in(0,1]$.
 
  We can ask if $q_\lambda$ is not even more regular, for example $\mathcal{C}^1$. In fact, if $q_\lambda$ is $\mathcal{C}^1$ and the measure of the points where the derivate vanishes is zero (which is generically true for a $\mathcal{C}^1$ function), the push-forward of the Lebesgue measure by $q_\lambda$ is absolutely continuous with respect to Lebesgue measure.  
 
If $q_\lambda$ is $\mathcal{C}^2$ with a finite number of points where the derivate vanishes (which is generically true for a $\mathcal{C}^2$ function), the number of poles of the push-forward of the Lebesgue measure by $q_\lambda$ is finite. The poles (which are singularities in $\frac{1}{\sqrt{x}}$) are located in the image of the points where the derivate vanishes by $q_\lambda$ (compare Figure~\ref{figure2}
and Figure~\ref{figure3}) . The hypothesis that the function is more than Hölder continuous but $\mathcal{C}^1$, and even perhaps $\mathcal{C}^k$ (for $k\geq 1$), seems possible because this is the case in other similar models, see for example \cite{MR1898800} in hyperbolic systems.
\newpage
\appendix
\section{Annex: Hölder spaces on a compact set}

In this section, let $(X, d)$ be a compact metric space and $(Y, \lVert \cdot \rVert)$ be a normed vector space.

\begin{definition}
    Let $f \in \mathcal{C}(X,Y)$ and $0 < \alpha \leq 1$. $f$ is said to be $\alpha$-Hölder if there exists $C \geq 0$ such that for all $x, y \in X$:
    \begin{align*}
        \lVert f(x) - f(y)\rVert \leq C d(x, y)^\alpha.
    \end{align*}
    We denote $\mathcal{C}^{0,\alpha}(X,Y)$ the space of $\alpha$-Hölder functions.
For $f\in\mathcal{C}^{0,\alpha}(X) $, let:
\begin{align*}
    |f|_{\alpha}\defeq \inf_{x,y\in X:x\neq y}\frac{\lVert f(x)-f(y)\rVert}{d(x,y)^\alpha},
\end{align*}
which is a seminorm on $\mathcal{C}^{0,\alpha}(X,Y)$.

For $f\in\mathcal{C}^{0,\alpha}(X) $, let
\begin{align*}
    \lVert f\rVert_{\alpha}\defeq \lVert f\rVert_{\infty}+|f|_{\alpha}
\end{align*}
which is a norm on $\mathcal{C}^{0,\alpha}(X,Y)$.
\end{definition}

Lemma~\ref{lem2.3} gives a condition that the limit inherits the Hölder continuity of a sequence of functions.

\begin{lemma}\label{lem2.2}
Let $\alpha\in (0,1]$ and $(f_n)_{n\in\mathbb{N}}\in\mathcal{C}^{0,\alpha}(X,Y)^\mathbb{N} $. Assume that $(f_n)_{n\in\mathbb{N}}$ converges pointwise to a function $f$ and $\underset{n\in\mathbb{N}}{\sup}~|f_n|_\alpha<+\infty$. Then $f\in\mathcal{C}^{0,\alpha}(X)$ and $|f|_\alpha\leq \underset{n\in\mathbb{N}}{\sup}~|f_n|_\alpha$.
\end{lemma}

\begin{proof}
    Let $C\defeq \underset{n\in\mathbb{N}}{\sup}~|f_n|_\alpha$. For $x,y\in\mathbb{X}$ and any $n\in\mathbb{N}$:
    \begin{align*}
        \lVert f_n(x)-f_n(y)\rVert\leq |f_n|_\alpha d(x,y)^\alpha\leq C d(x,y)^\alpha.
    \end{align*}
    By the triangle inequality:
    \begin{align*}
        \lVert f(x)-f(y)\rVert&\leq C d(x,y)^\alpha + \lVert f_n(x)-f(x)\rVert + \lVert f_n(y)-f(y)\rVert.
    \end{align*}
    Therefore, by the pointwise convergence:
    \begin{align*}
        \lVert f(x)-f(y)\rVert&\leq C d(x,y)^\alpha.
    \end{align*}
    Thus, $|f|_\alpha\leq C$.\qedhere
\end{proof}
Proposition~\ref{prop2.3} is a result of interpolation on Hölder spaces.
\begin{prop}\label{prop2.3}
       Let $0<\beta<\alpha\leq1$, $t=\frac{\beta}{\alpha} \in (0,1)$, and $f\in\mathcal{C}^{0,\alpha}(X,Y)$. Then $|f|_\beta\leq (2\lVert f\rVert_{\infty})^{1-t} |f|_\alpha^{t}$.  
\end{prop}

\begin{proof}
    For $x,y\in X$:
    \begin{align*}
        \lVert f(x)-f(y)\rVert&\leq 2 \lVert f\rVert_\infty\\
        \text{and }  \lVert f(x)-f(y)\rVert&\leq |f|_\alpha d(x,y)^\alpha.
    \end{align*}
    Therefore,
    \begin{align*}
        \lVert f(x)-f(y)\rVert&\leq (2\lVert f\rVert_\infty)^{1-t} |f|_\alpha^{1-t} d(x,y)^{t\alpha}=(2\lVert f\rVert_\infty)^{1-t} |f|_\alpha^{t} d(x,y)^\alpha.
    \end{align*}
    Thus, $|f|_\beta\leq (2\lVert f\rVert_\infty)^{1-t}|f|_\alpha^{t}$.
\end{proof}
\newpage
 \section*{Acknowledgments}
The author thanks his thesis supervisor, Damien Thomine, for suggesting the model, his advice, and proofreading. He also thanks the Academic Writing Center of Paris Saclay University for proofreading.
 
\bibliographystyle{alpha}
\bibliography{biblio.bib}

\newpage 

\begin{center}
    {\LARGE\bfseries Erratum : Galton-Watson processes in dynamical environments}
\end{center}

Marc Peigné, whom we thank, found an error in the proof of \cite[Theorem 1.3.1]{M1}. Indeed, we see this theorem as a consequence of \cite[Corollary 1 and Theorem 3]{MR0298780}. However, in \cite{MR0298780}, the authors suppose that for all environments the law of reproduction is not a convex combination of $\delta_0$ and $\delta_1$. In other words, the probability to have two or more offspring is positive for all environments. In our article, we only suppose that:\begin{itemize}
    \item $\mu_x\neq\delta_0$ for all $x\in\mathbb{X}$,
    \item there are no non-empty, closed, and $T$-invariant subsets of $\{x\in\mathbb{X}:\mu_x=\delta_1\}$.
\end{itemize}
We propose a correction to \cite[Theorem 1.3.1]{M1} by considering the induced process in which the hypotheses of \cite{MR0298780} hold.

\section*{The results}

We recall the result \cite[Theorem 1.3.1]{M1} for which the proof is incorrect.

\encad{
\begin{hyp}[H\ref{V2hyp1}]\label{V2hyp1}~
\begin{enumerate}[label=\alph*)]
    \item $x\in\mathbb{X}\mapsto \mu_x$ is continuous.
    \item $\mu_x\neq\delta_0$ for all $x\in\mathbb{X}$.
    \item There are no non-empty, closed, $T$-invariant subsets of $\{x\in\mathbb{X}:\mu_x=\delta_1\}$.\label{1c}
\end{enumerate}
\end{hyp}}\medskip

\begin{theorem}\cite[Theorem 1.3.1]{M1}\label{V2thm1}
    Assume (H\ref{V2hyp1}). Let $\nu\in\mathcal{E}_T(\mathbb{X})$. Then, $\nu(N)=0$ if and only if $\mathbb{E}_\nu[\log m(\cdot)]>0$.
\end{theorem}

\section*{Corrected Proof}

\subsection*{Constant-step acceleration: dynamic powers}

Let $n_0\in\mathbb{N}^*$. We consider the Galton-Watson process in dynamical environments given by $(T^{n_0},\varphi^{(n_0)})$. Dynamically, this amounts to speeding up the process. In terms of the Galton-Watson process, this amounts to considering only those generations that are multiples of $n_0$. The following lemma shows that we can study the process $(T^{n_0},\varphi^{(n_0)})$ to obtain information on the extinction of the process $(T,\varphi)$.

\begin{lemma}\label{extN}
The probability of extinction is the same for the Galton-Watson process in dynamical environments given by $(T^{n_0},\varphi^{(n_0)})$ and $(T,\varphi)$. In particular, the set of bad environments of these two Galton-Watson processes in dynamical environments is the same.
\end{lemma}

\begin{proof}
Let $q^{(n_0)}$ be the probability of extinction of $(T^{n_0},\varphi^{(n_0)})$. Let $x\in\mathbb{X}$. We have that,\begin{align*}
    q(x)&=\underset{n\to\infty}{\lim} \varphi^{(n)}(x,0)=\underset{n\to\infty}{\lim} \varphi^{(nn_0)}(x,0)=q^{(n_0)}(x).\qedhere
\end{align*}
\end{proof}

The benefit of switching to the process $(T^{n_0},\varphi^{(n_0)})$ is that it sometimes satisfies properties that are better than those of the original system. The following lemma is an example of this.

\begin{lemma}\label{lemmedelta1}
If a Galton-Watson process in dynamical environments $(T,\varphi)$ satisfies the hypothesis \textquotedblleft there are no non-empty, closed, and $T$-invariant subsets of $\{x\in\mathbb{X}:\mu_x=\delta_1\}$ \textquotedblright, then there exists $n_0\in\mathbb{N}^*$ such that the Galton-Watson process in dynamical environments $(T^{n_0},\varphi^{(n_0)})$ verifies the stronger hypothesis \textquotedblleft for all $x\in\mathbb{X}$, $\mu^{(n_0)}_x\neq \delta_1$\textquotedblright, (where $x\in\mathbb{X}\mapsto \mu^{(n_0)}_x $ is the law of reproduction of $(T^{n_0},\varphi^{(n_0)})$).
\end{lemma}

\begin{proof}
If there exists $x\in\mathbb{X}$ such that $\mu_{T^nx}=\delta_1$ for all $n\in\mathbb{N}$, then $\overline{\{T^kx:k\in\mathbb{N}\}}$ is a non-empty, closed, $T$-invariant subset of $\{x\in\mathbb{X}:\mu_x=\delta_1\}$ which contradicts the hypothesis. 
For all $n\in\mathbb{N}$, let\begin{align*}
        K_n\defeq\left\{x\in\mathbb{X}: \mu_{T^kx}=\delta_1 \text{ for all } k\in\llbracket0,n-1\rrbracket\right\}.
    \end{align*} 
Since $\mu$ and $T$ are continuous, $K_n$ is compact for all $n\in\mathbb{N}$. Moreover, the sequence of compact sets $(K_n)_{n\in\mathbb{N}}$ is decreasing and $\underset{n\in\mathbb{N}^*}{\bigcap} K_n=\emptyset$. Therefore, there exists $n_0\in\mathbb{N}^*$ such that $K_{n_0}=\emptyset$. Thus, for all $x\in\mathbb{X}$, $\mu^{(n_0)}_x\neq \delta_1$.
\end{proof}

\subsection*{Acceleration by return time: induced process}

Let $\nu\in\mathcal{E}_T(\mathbb{X})$ and $A\in\mathcal{B}(\mathbb{X})$ be such that $\nu(A)>0$. Let $\nu_A\in\mathcal{P}(A)$ be the restriction of $\nu$ to $A$ normalized to be a probability. 

The first return time to $A$ is the map $\tau_A$ from $A$ to $\mathbb{N}^*\cup\{+\infty\}$ defined for all $x\in A$ by: \begin{align*}
    \tau_A(x)=\min\{n\in\mathbb{N}^*:T^n(x)\in A\}.
\end{align*}
By Poincaré's recurrence theorem, $\tau_A$ is finite for $\nu$-almost all $x\in A$.
We define the induced transformation of $T$ on $A$ by, for all $x\in A$, \begin{align*}
T_A(x)\defeq\left\{\begin{array}{lll}
    T^{\tau_A(x)}x&\text{if} &\tau_A(x)<+\infty\\
     x&\text{else} 
    \end{array}\right. ,
     \end{align*}

We define the probability generating function by, for all $(x,s)\in A\times [0,1]$ \begin{align*}
\varphi_A(x,s)\defeq\left\{\begin{array}{lll}
    \varphi^{(\tau_A(x))}(x,s)&\text{if} &\tau_A(x)<+\infty\\
     \varphi(x,s)&\text{else} 
    \end{array}\right. .
     \end{align*} 
Let $\mu^{(A)}$ be the associated law of reproduction. 

In this case, the transformation $T_A$ and the law of reproduction $\mu^{(A)}$ are not necessarily continuous.  Nevertheless, for each $x\in A$, let $(Z^A_n(x))_{n\in\mathbb{N}}$ be the Galton–Watson process in varying environments associated with the sequence of laws of reproduction $(\mu^{(A)}_{T_A^n x})_{n\in\mathbb{N}}$. If we choose the initial environment $x$ according to $\nu_A$, then $\mu^{(A)}_{T_A^n x}$ is a stationary and ergodic process that allows us to use Athreya and Karlin's results \cite{MR0298780}.

We denote by $q^{(A)}$ the probability of extinction of $(T^{A},\varphi^{(A)})$, and by $N_A$ the associated set of bad environments.

\begin{lemma}
For $\nu$- (or $\nu_A$-) almost all $x\in A$, \begin{align*}
    q(x)=q_A(x).
\end{align*}
In particular, $N_A=N\cap A$ (up to a set of $\nu$-measure zero).
\end{lemma}

\begin{proof}
 For all $x\in A\setminus E$, for all $n\in\mathbb{N}^*$, \begin{align*}
    \varphi^{(\tau_A(x)+\cdots+\tau_A(T_A^{n-1}x))}(x,s)=\varphi_A^{(n)}(x,0).
\end{align*}
Taking the limit as $n\to +\infty$, the conclusion follows for all $x\in A\setminus E$, which is a set of full $\nu_A$-measure.
\end{proof}

The following lemma allows us to connect the systems $(T,\varphi)$ and $(T^{A},\varphi^{(A)})$.

\begin{lemma}\label{logmA}
For $\nu$- (or $\nu_A$-) almost all $x\in A$, \begin{align*}
    \log(\partial_s\varphi_A(x,1))=\sum_{i=0}^{\tau_A(x)-1}\log m(T^ix).
\end{align*}
\end{lemma}

\begin{proof}
The conclusion follows by \cite[Corollary 2.1.3]{M1} for all $x\in  A\setminus E $.
\end{proof}

\subsection*{Proof of the theorem}

\begin{proof}[Proof of Theorem~\ref{V2thm1}]
\underline{First step:} We will start by proving the theorem under the additional assumption that $\mu_x\neq \delta_1$ for all $x\in\mathbb{X}$.

Let $A=\{x\in\mathbb{X}:\mu_x(\{0,1\})<1\}$ and $\nu\in\mathcal{E}_T(\mathbb{X})$.

Suppose that $\nu(A)=0$. Then, for $\nu$ almost all $x\in\mathbb{X}$,\begin{align}\label{eqx}
    \mu_x(\{0,1\})=1.
\end{align}
By the continuity of $\mu$, there exists $C>0$ such that \begin{align}\label{eqc}
    \mu_x(1)<1-C.
\end{align} 
Thus, by \eqref{eqx} and \eqref{eqc}, $\nu(N)=1$. Moreover, by \eqref{eqx} and by Birkhoff's ergodic theorem, \begin{align*}
    \mathbb{E}_\nu[\log m(\cdot)]\leq 0.
\end{align*} 

Suppose now that $\nu(A)>0$. We'll start by showing that, for the induced transformation, the assumptions of \cite[Corollary 1 and Theorem 3]{MR0298780} hold.

By definition of the set $A$, for all $x\in A$, $\mu_x(\{0,1\})<1$. We recall that $x\in\mathbb{X}$, $\mu_x\neq\delta_0$ and $\mu_x\neq\delta_1$. For all $x\in A$, the function $s\mapsto\varphi_A(x,s)$ is not affine, as it is a composition of non-constant affine functions and the function $s\mapsto\varphi(x,s)$, which has an order greater than or equal to two. In other words, the law of reproduction of the induced process is not a convex combination of $\delta_0$ and $\delta_1$, which is a hypothesis of \cite[Corollary 1 and Theorem 3]{MR0298780} that is not verified by the process $(T,\varphi)$.

Since for all $x\in\mathbb{X}$, $\mu_x\neq\delta_0$ and $\mu$ is continuous, there exists $C>0$ such that $1-\mu_x(0)\geq C$ for all $x\in\mathbb{X}$.
Let $x\in A\setminus E$ (where $E=\uo{n=0}{+\infty}{\bigcup}T_A^{-n}\{x\in A:
\tau_A(x)=+\infty\}\subset A)$. Then, \begin{align*}
    -\log (1-\varphi_A(x,0))=-\log(1-\mu^{(A)}_x(0))\leq -\tau_A(x)\log(C).
\end{align*}
As $A\setminus E$ is of $\nu_A$ full measure, then\begin{align}\label{eqAK1}
    \mathbb{E}_{\nu_A}[-\log(1-\varphi_A(\cdot,0))]\leq -\log(C)\int_{A}\tau_A(x)\, \mathrm{d}\nu_A(x)<+\infty,
\end{align} by Kac's theorem.

Moreover, for all  $x\in A\setminus E$, by Lemma~\ref{logmA},
\begin{align*}
    \log(\partial_s\varphi_A(x,1))=\sum_{i=0}^{\tau_A(x)-1}\log m(T^ix)\geq \sum_{i=0}^{\tau_A(x)-1}\log (1-\mu_{T^ix}(0))\geq \tau_A(x)\log(C).
\end{align*}

As shown above, we obtain that:
\begin{align}\label{eqAK2}
   \mathbb{E}_{\nu_A}[-\log(1-\varphi_A(\cdot,0))]^-<+\infty 
\end{align}

We have verified via \eqref{eqAK1} and \eqref{eqAK2} that the two integrability hypotheses required for the application of \cite[Corollary 1 and Theorem 3]{MR0298780} hold. Thus, we obtain that
\begin{align}\label{eqthm1}
   \nu_A(N_A)=0 \qquad \text{if and only if}\qquad \mathbb{E}_{\nu_A}[\log(\partial_s\varphi_A(\cdot,1))]>0. 
\end{align}

Now we need to return to the original system.

\begin{itemize}
    \item Since $N_A=N\cap A$ (up to a set of $\nu$-measure zero), $\nu(A)>0$ and $\nu_A(N_A),\nu(N)\in\{0,1\}$, \begin{align}
        \nu_A(N_A)=0 \qquad \text{if and only if} \qquad \nu(N)=0.\label{eqthm2}
    \end{align}
    \item By Kac's Formula and by Lemma~\ref{logmA}, \begin{align*}
        \int_{\mathbb{X}}\log m(x)\, \mathrm{d}\nu(x)
        &=\int_{A}\sum_{i=0}^{\tau_A(x)-1}\log m(T^ix)\, \mathrm{d}\nu(x)\\
        &= \int_{A}\log(\partial_s\varphi_A(x,1))\, \mathrm{d}\nu(x)\\
        &= \nu(A) \int_{A}\log(\partial_s\varphi_A(x,1))\, \mathrm{d}\nu_A(x).
    \end{align*}
    Thus, \begin{align}
        \mathbb{E}_{\nu_A}[\log(\partial_s\varphi_A(\cdot,1))]>0 \qquad \text{if and only if} \qquad \mathbb{E}_\nu[\log m(\cdot)]>0.\label{eqthm3}
    \end{align}
    The conclusion of the first step follows by Equivalences~\eqref{eqthm1}, \eqref{eqthm2}, and \eqref{eqthm3}.
\end{itemize}

\underline{Second step:}
By Lemma~\ref{lemmedelta1} there exists $n_0\in\mathbb{N}^*$ such that the Galton-Watson process in dynamical environments $(T^{n_0},\varphi^{(n_0)})$ verifies the stronger hypothesis \textquotedblleft for all $x\in\mathbb{X}$, $\mu^{(n_0)}_x\neq \delta_1$\textquotedblright. We obtain (by the first step) that for all $\nu\in\mathcal{E}_{T^{n_0}}(\mathbb{X})$,
    \begin{align}\label{eqitere}
        \nu(N)=0 \qquad\text{if and only if}\qquad \mathbb{E}_\nu[\log \partial_s\varphi^{(n_0)}(\cdot,1)]>0,
    \end{align} 
    because (Lemma~\ref{extN}) the set of bad environments is the same for systems $(T,\varphi)$ and $(T^{n_0},\varphi^{(n_0)})$.
    
    Let $\nu\in\mathcal{E}_{T}(\mathbb{X})$. By Choquet's integral representation theorem~\cite{SB_1956-1958__4__33_0}, there exists a probability measure $P$ on $\mathcal{E}_{T^{n_0}}(\mathbb{X})$ such that: \begin{align*}
    \nu=\int_{\mathcal{E}_{T^{n_0}}(\mathbb{X})}\widetilde{\nu}\, \mathrm{d}P(\widetilde{\nu}).
\end{align*}
Moreover, as $\nu$ is $T$-invariant, \begin{align}
    \mathbb{E}_\nu[\log m(\cdot)]&=\frac{1}{n_0}\mathbb{E}_\nu\left[\sum_{k=0}^{n_0-1}\log m\circ T^k(\cdot)\right]\nonumber\\
    &=\frac{1}{n_0}\mathbb{E}_\nu\left[\log \partial_s\varphi^{(n_0)}(\cdot,1)\right]\nonumber\\
    &= \frac{1}{n_0}\int_{\mathcal{E}_{T^{n_0}}(\mathbb{X})} \mathbb{E}_{\widetilde{\nu}}[\log \partial_s\varphi^{(n_0)}(\cdot,1)] \, \mathrm{d}P(\widetilde{\nu}).\label{eqdec}
\end{align}
By \cite[Proposition 2.3.1]{M1}, we know that $\nu(N)\in\{0,1\}$. \begin{itemize}
    \item If $\nu(N)=0$, then\begin{align*}
        \int_{\mathcal{E}_{T^{n_0}}(\mathbb{X})}\widetilde{\nu}(N)\, \mathrm{d}P(\widetilde{\nu})=0.
    \end{align*}  It follows that $\widetilde{\nu}(N)=0$ for $P$-almost every $\widetilde{\nu} \in \mathcal{E}_{T^{n_0}}(\mathbb{X})$ (because $\widetilde{\nu}(N)\in\{0,1\}$).  According to \eqref{eqitere}, this implies $\mathbb{E}_{\widetilde{\nu}}[\log \partial_s\varphi^{(n_0)}(\cdot,1)] > 0$ $P$-almost surely. We conclude by~\eqref{eqdec} that:
    \begin{align*}
        \mathbb{E}_\nu[\log m(\cdot)]=\frac{1}{n_0} \int_{\mathcal{E}_{T^{n_0}}(\mathbb{X})} \mathbb{E}_{\widetilde{\nu}}[\log \partial_s\varphi^{(n_0)}(\cdot,1)] \, \mathrm{d}P(\widetilde{\nu}) > 0.
    \end{align*}
     \item If $\nu(N)=1$, then\begin{align*}
        \int_{\mathcal{E}_{T^{n_0}}(\mathbb{X})}\widetilde{\nu}(N)\, \mathrm{d}P(\widetilde{\nu})=1.
    \end{align*}  It follows that $\widetilde{\nu}(N)=1$ for $P$-almost every $\widetilde{\nu} \in \mathcal{E}_{T^{n_0}}(\mathbb{X})$.  According to \eqref{eqitere} again, this implies $\mathbb{E}_{\widetilde{\nu}}[\log m(\cdot)] \leq 0$ $P$-almost surely. We conclude again by~\eqref{eqdec} that:
    \begin{align*}
        \mathbb{E}_\nu[\log m(\cdot)]&=\frac{1}{n_0} \int_{\mathcal{E}_{T^{n_0}}(\mathbb{X})} \mathbb{E}_{\widetilde{\nu}}[\log \partial_s\varphi^{(n_0)}(\cdot,1)] \, \mathrm{d}P(\widetilde{\nu})\leq 0.\qedhere
    \end{align*}
\end{itemize}
\end{proof}

\end{document}